\documentclass[12pt]{article}
\usepackage[left=2.5cm, right=2.5cm, bottom=2.5cm, top=3cm, headsep=3mm]{geometry}

\usepackage{amsmath}
\usepackage{amsfonts}
\usepackage{verbatim}
\usepackage{amsthm}
\usepackage{amssymb}
\usepackage{epsfig}
\usepackage{amsrefs}
\usepackage{xcolor}

\newtheorem{theorem}{Theorem}[section]
\newtheorem{lemma}[theorem]{Lemma}

\newcounter{claim_nb}[theorem]
\newtheorem{claim}[claim_nb]{Claim}

\newenvironment{cproof}
{\begin{proof}
 [Proof of claim.]
 \vspace{-1.2\parsep}}
{ \end{proof}}

\newcommand\lref[1]{Lemma~\ref{lem:#1}}
\newcommand\tref[1]{Theorem~\ref{thm:#1}}
\newcommand\sref[1]{Section~\ref{sec:#1}}
\newcommand\clref[1]{Claim~\ref{cl:#1}}

\newcommand{\del}{\backslash}

\newcommand{\bs}[0]{\backslash}
\newcommand{\bal}{\mathcal{B}}
\newcommand{\zA}{\mathcal A}
\newcommand{\proj}{\mathtt{Proj}}

\renewcommand{\MR}[1]{}

\title{Biased graphs with no two vertex-disjoint unbalanced cycles}

\author{Rong Chen\footnote{Center for Discrete Mathematics, Fuzhou University, Fuzhou, Fujian, China. Email: rongchen@fzu.edu.cn. The author is supported partially by CNNSF (No.11201076)  and CSC.},
Irene Pivotto\footnote{School of Mathematics and Statistics, University of Western Australia, Perth WA, Australia. Email: irene.pivotto@uwa.edu.au.
This author is supported by an Australian Research Council Discovery Project (project number DP110101596).}}

\begin{document}
\maketitle

\begin{abstract}
Lov\'asz has completely characterised the structure of graphs with no two vertex-disjoint cycles, while
Slilaty has given a structural characterisation of graphs with no two vertex-disjoint odd cycles; his result is in fact more general, describing signed graphs with no two vertex-disjoint negative cycles.
A {\em biased graph} is a graph with a distinguished set of cycles (called balanced) with the property that any theta subgraph does not contain exactly two balanced cycles. In this paper we characterise the structure of biased graphs with no two vertex-disjoint unbalanced cycles, answering a question by Zaslavsky and generalising the results of Lov\'asz and Slilaty. 
\end{abstract}

\section{Introduction}

By a {\em cycle} in a graph we mean a connected subgraph where every vertex has degree two.
Throughout the paper we will say that two subgraphs are {\em disjoint} to mean that they are vertex-disjoint; this applies in particular to cycles and paths.
A {\em biased graph} is a pair $(G, \bal)$, where $G$ is a graph and $\bal$ is a collection of cycles of $G$ satisfying the {\em theta property}, which is as follows. For any two cycles $C_1$ and $C_2$ in $\bal$ such that $C_1 \cap C_2$ is a path with at least one edge, the third cycle in $C_1 \cup C_2$ is also in $\bal$. The cycles in $\bal$ are called {\em balanced}, while those not in $\bal$ are {\em unbalanced}. Biased graphs were introduced by Zaslavsky in~\cite{MR1007712}. Examples of biased graphs are graphs with all cycles balanced, graphs with all cycles unbalanced and biased graphs arising from group-labelled graphs. In a group-labelled graph each edge is oriented and assigned a value from a group; a cycle is balanced if multiplying the group values along the cycle (where we take the inverse values on edges traversed backwards) produces the group identity.

Biased graphs give rise to two main types of matroids, frame matroids and lift matroids (see~\cite{MR1088626}). We will not discuss these matroids here, but merely mention that these two matroids are the same, for a given biased graph $\Omega$, if and only if $\Omega$ does not contain two vertex-disjoint unbalanced cycles. Hence the question arises of which biased graphs have this property.
This question was first posed by Zaslavsky (Problem~3.5 in~\cite{MR1088626}) and is the subject of this paper.

There are two simple cases of biased graphs having no two vertex-disjoint unbalanced cycles. The first is biased graphs with no unbalanced cycles at all, i.e. biased graphs of the form $(G,\bal)$, where $\bal$ is the set of all cycles of $G$. Biased graphs of this form are called {\em balanced}. The second simple example of biased graphs with no two disjoint unbalanced cycles is biased graphs where all unbalanced cycles use a specific vertex $v$, which is then called a {\em blocking vertex}. We will focus on biased graphs that have no two vertex-disjoint unbalanced cycles, but are not balanced and have no blocking vertex. Such biased graphs are called {\em tangled}.

A special type of biased graphs are those where all cycles are unbalanced. In this case, our question reduces to ask for the structure of graphs with no two vertex-disjoint cycles. This question was answered by Lov\'asz in~\cite{MR0211902} (see~\cite{MR2078877} for a proof in English).

\begin{theorem}[Lov\'asz~\cite{MR0211902}]\label{thm:Lovasz}
Let $G$ be a connected graph with no two disjoint cycles. Then either $G-v$ is a forest for some $v \in V(G)$, or $G$ is a subgraph of a graph obtained from the ones in Figure~\ref{fig:Lovasz} by possibly attaching trees on single vertices (where, in the figure, $k \geq 1$ and $\ell \geq 3$).
\end{theorem}

\begin{figure}[htbp]
\begin{center}
\includegraphics[page=1,height=3.5cm]{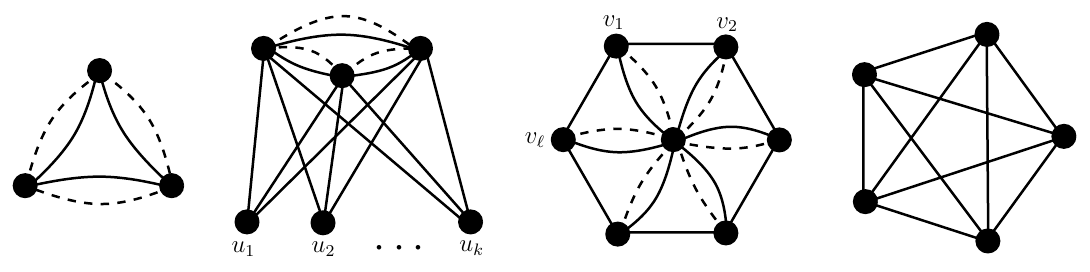}
\caption{Graphs with no two disjoint cycles. A dotted edge indicates that any number of parallel edges may be added to that edge.}
\label{fig:Lovasz}
\end{center}
\end{figure}

Another particular type of biased graphs are those arising from signed graphs. A signed graph is a pair $(G,S)$, where $S \subseteq E(G)$. The associated biased graph is $(G,\bal_S)$, where a cycle $C$ is in $\bal_S$ if and only if $|C\cap S|$ is even. If $S$ is replaced by $S'=S\Delta D$, for an edge cut $D$ of $G$, then $\bal_S=\bal_{S'}$. For simplicity we will sometimes identify the biased graph arising from a signed graph with the signed graph itself.
A family of tangled signed graphs is that of {\em projective planar} signed graphs, which are signed graphs of the form $(G,S)$, where $G$ can be embedded in the projective plane so that $S$ is a nonseparating cycle of the topological dual of $G$. In other words, $G$ may be embedded in the projective plane so that the unbalanced cycles are exactly the nonseparating cycles.
In~\cite{MR2344133} Slilaty characterized tangled signed graphs and showed that, saved for a specific case and simple decompositions, they are projective planar. We will make use of similar decompositions for biased graphs; such decomposition will be discussed in~\sref{separations}.

\begin{theorem}[Slilaty~\cite{MR2344133}]\label{thm:Slilaty}
Any connected tangled signed graph is either
\begin{itemize}
    \item projective planar, or
    \item isomorphic to $(G, E(G))$, along with possibly some balanced loops, where $G$ is obtained from $K_5$ by adding parallel edges, or
    \item a $1$-, $2$- or $3$-sum of a tangled signed graph and a balanced signed graph having at least $2, 3$ or $5$ vertices respectively.
\end{itemize}
\end{theorem}

If a signed graph $(G,S)$ is taken with $S=E(G)$, then the unbalanced cycles of $(G,S)$ are exactly the odd cycles of $G$. Thus \tref{Slilaty} also describes graphs having no two vertex-disjoint odd cycles.
Such characterisation was also given for internally 4-connected graphs in~\cite{MR3048156}. 

Our main result is the proof of the following theorem, which generalises both \tref{Lovasz} and \tref{Slilaty}. We make use of \tref{Lovasz} in our proof (specifically, in the proof of \lref{balanced-sub}), while \tref{Slilaty} follows easily from our result.

We say that a biased graph is {\em simple} if it does not contain balanced loops and pairs of parallel edges $e$ and $f$ such that $\{e,f\}$ is a balanced cycle. The {\em simplification} of a biased graph $\Omega$ is a maximal subgraph of $\Omega$ which is simple. Thus the simplification of $\Omega$ is the biased graph  obtained by deleting all balanced loops and all but one edge in any balanced parallel class. A biased graph is tangled if and only if its simplification is tangled. Thus we only consider simple biased graphs in our result. If $\Omega=(G,\bal)$ is a biased graph, we denote by $||\Omega||$ the graph $G$ (called the {\em underlying graph} of $\Omega$).

\begin{theorem}\label{thm:main-thm}
Let $\Omega$ be a simple connected tangled biased graph. Then either:
\begin{itemize}
    \item[(T1)] $\Omega$ is one of the following:
    \begin{itemize}
        \item[(T1a)] a projective planar signed graph,
        \item[(T1b)] a generalised wheel,
        \item[(T1c)] a criss-cross,
        \item[(T1d)] a fat triangle,
        \item[(T1e)] projective planar with a special vertex,
        \item[(T1f)] projective planar with a special pair,
        \item[(T1g)] projective planar with a special triple,
        \item[(T1h)] a tricoloured graph, or
    \end{itemize}
    \item[(T2)] $||\Omega||$ is obtained from $K_5$ by possibly adding edges in parallel to an edge of $K_5$, or
    \item[(T3)] $\Omega$ is a $1$-, $2$- or $3$-sum of a tangled biased graph and a balanced signed graph having at least $2, 3$ or $4$ vertices respectively.
\end{itemize}
\end{theorem}

The structures in (T1b)-(T1h) are described in \sref{structures}; all these structures, except for the generalised wheel and fat triangle, occur when the underlying graph of $\Omega$ is projective planar.
Somewhat surprisingly, all of (T1b)-(T1h) have the property that the removal of some set of at most three vertices leaves a balanced graph.

In the next section we provide basic definitions that will be used throughout the paper.
The structures in \tref{main-thm} are defined in \sref{structures}; figures for these structures may be found in the Appendix.
For the proof of the main theorem we need to use and extend existing results on linkages; these are presented in \sref{linkages}.
In \sref{separations} we consider the case when the tangled biased graph has small separations; we show that in this case $\Omega$ is either decomposable along a $1$-, $2$- or $3$-sum or $\Omega$ is a generalised wheel.
In \sref{balanced_sub} we show that $\Omega$ contains a  2-connected spanning balanced subgraph, unless $\Omega$ is a criss-cross or
a projective planar biased graph with a special pair. In the same section we also show that such balanced subgraph may be chosen to be planar, unless $\Omega$ is a fat triangle.
Finally, \sref{main_proof} contains the proof of \tref{main-thm}.

\section{Basic definitions}\label{sec:definitions}

All graphs in this work are undirected and may have loops and parallel edges.
Let $G$ be a graph.  For a set $X$ of vertices we denote by $G[X]$ the subgraph of $G$ induced by $X$ and by
$G-X$ the subgraph $G[V(G)-X]$; we use $G-v$ as shorthand notation for $G-\{v\}$.
Moreover, we denote by $N_G(X)$ the set of vertices that are not in $X$ but are adjacent to a vertex in $X$.
If $Y$ is a set of edges, we denote by $G[Y]$ the subgraph of $G$ induced by $Y$ and by $V_G(Y)$ the vertex set of $G[Y]$;
we let $G-Y$ denote the graph obtained from $G$ by deleting the edges in $Y$.

Given a graph $G$ and $X\subseteq V(G)$, $\delta_G(X):=\{uv\in E(G):u\in X,v\not\in X\}$
and we write $\delta_G(v)$ for $\delta_G(\{v\})$.
Throughout the paper we shall omit indices when there is no ambiguity.
For instance we may write $\delta(v)$ for $\delta_G(v)$.
Two edges in a graph are {\em independent} if they have no common endpoint. A set of edges $U$ is {\em independent} if the edges in $U$ are pairwise independent.

Let $G$ be a graph and $A$ and $B$ be sets of vertices of $G$.
An {\em $(A,B)$-path} is a path of $G$ with one endpoint in $A$ and one endpoint in $B$, and no other vertex in $A \cup B$.
We use ``$(a,b)$-path" as shorthand for ``$(\{a\},\{b\})$-path" and similarly,
``$(a,B)$-path" as shorthand for ``$(\{a\},B)$-path".
If $A=B$, then we refer to an $(A,A)$-path simply as an $A$-path.
Sometimes we abuse notation and in the previous definition we replace one or both of $A$ and $B$ with subgraphs of $G$. So an $(H_1,H_2)$-path (for subgraphs $H_1$ and $H_2$ of $G$) is simply a $(V(H_1),V(H_2))$-path.
A {\em theta graph} is a graph formed by three internally disjoint $(a,b)$-paths, for some distinct vertices $a,b$.

If $X$ is a set of edges of $G$, we define the {\em boundary} of $X$ as $V_G(X)\cap V_G(\bar{X})$ (where $\bar{X}$ denotes the complement of $X$) and the {\em interior} of $X$ as the vertices in  $V_G(X)$ that are not on the boundary. We also define the boundary and interior of a subgraph $H$ to be the boundary and interior of $E(H)$.
A partition $(A_1,A_2)$ of $E(G)$ is a {\em $k$-separation} if $G[A_1]$ and $G[A_2]$ are both connected with  nonempty interior and the boundary of $A_1$ has size $k$. Sometimes we will abuse notation and say that $(G[A_1], G[A_2])$ is a $k$-separation.
We may also omit one side of a $k$-separation and say, for example, that $A$ is a $k$-separation if $(A,\bar{A})$ is a $k$-separation.
A graph $G$ is {\em $k$-connected} if it has no $\ell$-separation for $\ell < k$.

A set of vertices $X$ of $G$ is a {\em vertex-cut} if $G- X$ is disconnected and $X$ is minimal with this property. It is a {\em $k$-vertex-cut} if it is a vertex-cut of size $k$.
If $X=\{v\}$ is a $1$-vertex-cut, we call $v$ a {\em cutvertex}.
A {\em bridge} of a vertex set $X$ is the subgraph of $G$ formed by a component $H$ of $G- X$ together with the edges between $H$ and $X$ and the endpoints of these edges. We also call a bridge of $X$ an $X$-bridge. Note that our definition of a bridge differ slightly from the usual one, since we do not consider an edge connecting two vertices of $X$ to be an $X$-bridge. 

Given a graph $G$, the {\em blocks} of $G$ are the maximal $2$-connected subgraphs of $G$. If $H_1,\ldots,H_k$ are all the blocks of $G$, then $E(H_1),\ldots,E(H_k)$ is a partition of $E(G)$ and we may associate a tree $\mathcal{T}$ with this partition. Let $\{v_1,\ldots,v_{\ell}\}$ be the set of cut-vertices of $G$ and define $V(\mathcal{T})=\{H_1,\ldots,H_k\} \cup \{v_1,\ldots,v_{\ell}\}$; block $H_i$ and cut-vertex $v_j$ are adjacent in $\mathcal{T}$ if $v_j \in V(H_i)$. A block is a {\em leaf block} if it corresponds to a leaf of $\mathcal{T}$.


Let $\Omega=(G,\bal)$ be a biased graph. The graph $G$ is the {\em underlying} graph of $\Omega$, denoted as $||\Omega||$.
We will often refer to properties of $||\Omega||$ as being properties of $\Omega$; for example, we may say that $\Omega$ is $k$-connected to mean that $||\Omega||$ is $k$-connected and we may write $\delta_{\Omega}(v)$ to mean $\delta_{||\Omega||}(v)$.
We say that a biased graph $\Omega'=(G',\bal')$ is a {\em subgraph} of $\Omega$ if $G'$ is a subgraph of $G$ and $\bal'=\{B \in \bal : B \subseteq G'\}$.
Given a set $X$ of edges of $G$, the biased graph {\em induced by $X$} is the subgraph of $\Omega$ with underlying graph $G[X]$. We will denote such subgraph as $\Omega[X]$.
When referring to a subgraph of $G$, we will assume that such subgraph inherits the structure of balanced cycles of $\Omega$. For example, when referring to a bridge of a set $X \subseteq V(G)$, we will often consider such bridge as a biased graph.

We say that two cycles in a biased graph {\em have the same bias} if they are both balanced or both unbalanced.
Let $C_1$ and $C_2$ be cycles such that $C_1\cup C_2$ is a theta subgraph. Let $C_3$ be the third cycle contained in $C_1\cup C_2$. We say that $C_3$ is obtained from $C_1$ by {\em rerouting along $C_2$}. If $C_1,\ldots, C_k$ is a sequence of cycles such that $C_{i+1}$ is obtained from $C_{i}$ be rerouting along some cycle $C^i$, then we say that $C_k$ is obtained from $C_1$ by {\em rerouting along $C^1,\ldots,C^{k-1}$}. By the theta property, if $C_2$ is obtained from $C_1$ by rerouting along a balanced cycle, then $C_1$ and $C_2$ have the same bias. Inductively, this is also the case if $C_2$ is obtained from $C_1$ by rerouting along a set of balanced cycles.
We will make repeated use of this fact throughout the paper.

Let $\Omega$ be a biased graph and $\Omega'$ be a balanced subgraph of $\Omega$.
Given a set $A \subseteq  E(\Omega)-E(\Omega')$, we say that a cycle $C$ of $\Omega$ is an $A$-cycle for $\Omega'$ if $A \subseteq C \subseteq \Omega' \cup A$. We write $e$-cycle as a shorthand for $\{e\}$-cycle.
We say that $F\subseteq E(\Omega)-E(\Omega')$ is {\em 2-balanced} with respect to $\Omega'$ if, for all $A \subseteq F$ with $|A|=2$, every $A$-cycle is balanced. The theta property implies that if $f \in E(\Omega)-E(\Omega')$ and some $f$-cycle for $\Omega'$ is unbalanced, then so are all the $f$-cycles for $\Omega'$ (see Proposition 3.1 in~\cite{MR1007712}). The same holds for $A$-cycles if $A$ is a set of two edges sharing an endpoint. However, this is not true in general: as an example, choose $||\Omega||=K_4$ and let $\Omega'$ be a 4-cycle of $\Omega$. Let $f_1$ and $f_2$ be the diagonals of this 4-cycle. Then we may assign one of the 4-cycles using $f_1$ and $f_2$ to be balanced, and the other to be unbalanced (while all the triangles are unbalanced).
We make use of 2-balanced sets in \lref{2balanced}: suppose that $\Omega$ is a connected tangled biased graph and that $\Omega'$ is a maximal balanced subgraph of $\Omega$. Then we show in the lemma that
if $E(\Omega)-E(\Omega')$ is 2-balanced with respect to $\Omega'$ then $\Omega$ is a signed graph (with signature $E(\Omega)-E(\Omega'))$.

A vertex $v$ of a biased graph $\Omega$ that intersects all unbalanced cycles of $\Omega$ is called a {\em blocking vertex}.
Two vertices $v$ and $w$ (neither of which is a blocking vertex) form a {\em blocking pair} if they intersect all unbalanced cycles.
Suppose that $v$ is a blocking vertex of $\Omega$ and $\Omega-v$ is connected.
In this case we define a relation $\sim_v$ on the edges in $\delta_{\Omega}(v)$ by declaring $e \sim_v f$ if either $e=f$ or all cycles containing $e$ and $f$ are balanced. This is an equivalence relation, as we show next.
Let $e_1,e_2,e_3$ be distinct edges in $\delta_{\Omega}(v)$ with $e_1 \sim_v e_2$ and $e_2 \sim_v e_3$. Let $H$ be a theta subgraph of $\Omega$ containing all of $e_1, e_2$ and $e_3$; such theta subgraph exists because $\Omega-v$ is connected, thus it contains a spanning tree. The cycle in $H$ containing both $e_1$ and $e_2$ is balanced, and so is the cycle containing both $e_2$ and $e_3$. Therefore the cycle $C$ containing $e_1$ and $e_3$ is balanced. Any other cycle containing $e_1$ and $e_3$ may be obtained from $C$ by rerouting along balanced cycles (contained in $\Omega - v$), hence all the cycles containing $e_1$ and $e_3$ are balanced and $e_1 \sim_v e_3$, showing that $\sim_v$ is an equivalence relation. The same argument shows that a cycle of $\Omega$ (that is not a loop) is unbalanced if and only if it contains two edges in $\delta_{\Omega}(v)$ which are not equivalent. We call the partition given by the equivalence classes of $\sim_v$ the {\em standard partition} of $\delta_{\Omega}(v)$.

\section{Tangled structures}\label{sec:structures}
In this section we describe the possible structure of tangled biased graph.
All the structures in this section are depicted in Appendix~\ref{sec:figures}.

We will make repeated use of the following definition (which will be repeated and extended in \sref{linkages}).
Given two disjoint sets of vertices $X$ and $Y$ in a graph $G$, we say that $(G, (X,Y))$ is {\em planar} if $G$ is a plane graph, $X\cup Y$ belongs to the same face $F$ of $G$ and there is some ordering $(x_1,\ldots,x_k)$ of the vertices in $X$ and some ordering $(y_1,\ldots,y_{\ell})$ of the vertices in $Y$ such that $x_1,\ldots,x_k,y_1,\ldots,y_{\ell}$ appear on $F$ in this circular order. If $X=\{x\}$ then we may abuse notation and write that $(G,(x,Y))$ is planar.
This definition extends to the cases when $x_k=y_1$ and/or $y_{\ell}=x_1$. We also extend this notation in the obvious way to the case when we have more than two sets.

\subsection{Generalized wheels}
Let $\Omega=(G,\bal)$ be a biased graph. Suppose that $G$ contains a special vertex $w$ such that:
\begin{itemize}
    \item[(a)]$G-w$ is the union of $2$-connected graphs $G_1,\ldots,G_k$ (for $k \geq 2$) with $V(G_i) \cap V(G_{i+1})=\{z_i\}$ for every $i \in [k]$
(where the indices are modulo $k$) and $z_1,\ldots,z_k$ are all distinct;
    \item[(b)] every $G_i$ is balanced;
    \item[(c)] every cycle of $G - w$ using all edges in $G_1,\ldots,G_k$ is unbalanced.
\end{itemize}
Moreover, for every $G_i$ that is not a single edge, the vertices in $(N_G(w) \cap V(G_i)) - \{z_{i-1},z_i\}$ partition into two nonempty sets $X_i$ and $Y_i$ such that:
\begin{itemize}
    \item[(d)] $(G_i,(z_{i-1},X_i,z_i,Y_i))$ is planar, and
    \item[(e)] for every pair of edges $wx$ and $wy$ with $x,y \in V(G_i)-\{z_{i-1},z_i\}$, and every $(x,y)$-path $P$ in $G_i$, the cycle $P \cup \{wx,wy\}$  is unbalanced if and only if one of $x$ and $y$ is in $X_i$ and the other is in $Y_i$.
\end{itemize}
The other cycles of $\Omega$ are not determined (as long as the theta property still holds).
We say that such $\Omega$ is a {\em generalized wheel}.
An example of a generalised wheel is given in Figure~\ref{fig:gen-wheel}.

\subsection{Criss-cross}
Starting from a planar graph $(H,(u_1,u_2,u_3,u_4))$, where $H$ is 2-connected, and a vertex $w$ not in $H$, we construct a tangled biased graph $\Omega$ as follows. The graph $||\Omega||$ is obtained from $H$ and $w$ by adding four edges $e_i=wu_i$, for $i\in [4]$, and two more edges $f_1=u_1u_3$ and $f_2=u_2u_4$.
Every cycle contained in $H$ is balanced.
Every $f_1$- and $f_2$-cycle for $H$ is declared unbalanced, and so are cycles formed by a $(u_i,u_j)$-path in $H$ together with $e_i$ and $e_j$, where $i\neq j$.
The two triangles $\{e_1,e_3,f_1\}$ and $\{e_2,e_4,f_2\}$ are balanced.
The other cycles of $\Omega$ are not determined (as long as the theta property still holds).
We call a biased graph constructed in this fashion a {\em criss-cross} (see Figure~\ref{fig:criss-cross}).

\subsection{Fat triangles}
Consider any graph $H$ with three distinct vertices $v_1,v_2,v_3$. Construct a biased graph $\Omega$ by adding nonempty sets of edges $F_{12}, F_{23}, F_{31}$, where every edge in $F_{ij}$ is between $v_i$ and $v_j$.
Declare $H$ to be balanced; every $f$-cycle is unbalanced, for all $f \in F_{12}\cup F_{23}\cup F_{13}$.
The other cycles of $\Omega$ are not determined (as long as the theta property still holds).
We call a biased graph constructed in this fashion a {\em fat triangle} (see Figure~\ref{fig:fat-triangle}).

\subsection{Projective planar with a special vertex}\label{sec:projective planar with a special vertex}
Consider two disjoint planar graphs $(H_1,(x_1,\ldots,x_m,u_1,z_2))$ and  $(H_2,(y_1,\ldots,y_m,z_1,u_2))$, where consecutive vertices in $x_1,\ldots,x_m$ and/or in $y_1,\ldots,y_m$ may be repeated. Construct a graph $H$ from $H_1$ and $H_2$ by adding edges $z_1z_2$ and $u_1u_2$ and adding a new vertex $w$ and edges $wz_1$ and $wz_2$.
We construct a biased graph $\Omega$ from $H$ as follows. The underlying graph $||\Omega||$ is obtained from $H$ by adding edges $g_1=wu_1$ and $g_2=wu_2$ and edges $f_i=x_iy_i$ for every $i \in [m]$. Denote by $F$ the set of edges $\{f_1,\ldots,f_m\}$.
Every cycle contained in $H$ is declared to be balanced.
For every $f \in F \cup \{g_1,g_2\}$, every $f$-cycle for $H$ is unbalanced.
Every $\{g_1,g_2\}$-cycle for $H$ is balanced and so is every $\{f_i,f_j\}$-cycle for $H-u_1u_2$, for all distinct $f_i,f_j \in F$.
Finally every $\{g_i,f_j\}$-cycle for $H-u_1u_2$ is unbalanced.
The bias of the other cycles in $\Omega$ may be chosen arbitrarily, as long as the theta property is preserved.
We call a biased graph constructed in this fashion a {\em projective planar biased graph with a special vertex} (see the left of Figure~\ref{fig:ppsv}).

\subsection{Projective planar with a special pair}
Consider a planar graph $(H,(x,y,X,Y))$, for some $X,Y \subseteq V(H)$ (sharing at most one vertex).
We construct a biased graph $\Omega$ from $H$ as follows.
The underlying graph  $||\Omega||$ is obtained from $H$ by adding the following edges:
\begin{itemize}
    \item edges $xx'$ for every $x' \in X$; denote by $F_x$ the set of edges added this way;
    \item edges $yy'$ for every $y' \in Y$; denote by $F_y$ the set of edges added this way;
    \item possibly adding edges $e_1,\ldots,e_{\ell}$ between $x$ and $y$.
\end{itemize}
We declare $H$ to be balanced, while $F_x$ and $F_y$ are 2-balanced for $H$.
Every $e_i$-cycle for $H$ is unbalanced. The bias of the other cycles in $\Omega$ may be chosen arbitrarily, as long as the theta property is preserved.
We call a biased graph constructed in this fashion a {\em projective planar biased graph with a special pair} (see the middle of Figure~\ref{fig:ppsv}).

\subsection{Projective planar with a special triple}\label{sec:ppst}
Consider a planar graph $(H,(y_1,x,y_2,X))$, for some $X \subseteq V(H)$.
We construct a biased graph $\Omega$ from $H$ as follows. The underlying graph  $||\Omega||$ is obtained from $H$ by adding the following edges:
\begin{itemize}
    \item edges $xx'$, for every $x' \in X$; denote by $F$ the set of edges added this way;
    \item edges $e_1,\ldots,e_{n}$ between $x$ and $y_1$;
    \item possibly edges $g_1,\ldots,g_{m}$ between $x$ and $y_2$;
    \item an edge $f=y_1y_2$.
\end{itemize}
We declare $H$ to be balanced, while $F$ is 2-balanced for $H$.
Every $e_i$-, $g_i$- and $f$-cycle for $H$  is unbalanced. The bias of the other cycles in $\Omega$ may be chosen arbitrarily, as long as the theta property is preserved.
We call a biased graph constructed in this fashion a {\em projective planar biased graph with a special triple} (see the right of Figure~\ref{fig:ppsv}).

\subsection{Tricoloured graphs}\label{sec:tricoloured}
All the indices in this definition are modulo 6 and either $I=\{1,2,3\}$ or $I=\{1,3,5\}$. Let $H$ be a 2-connected graph such that:
\begin{itemize}
    \item[(a)] $H$ is the union of connected graphs $H_1,\ldots,H_6$ with $V(H_i) \cap V(H_{i+1})=\{z_i\}$ for every $i \in [6]$
and $z_1,\ldots,z_6$ are all distinct.
    \item[(b)] For every $i \in I$, let $x_i$ be a vertex in $H_i$ and $Y_i$ be a set of vertices in $H_{i+3}$ such that
    \begin{itemize}
    \item If $I=\{1,2,3\}$, then $(H,(x_1,x_2,x_3,Y_1,Y_2,Y_3))$ is planar.
    \item If $I=\{1,3,5\}$, then $(H,(x_1,Y_5,x_3,Y_1,x_5,Y_3))$ is planar.
    \end{itemize}
    (Where we allow at most one of the $Y_i$'s to be empty.)
\end{itemize}
We construct a biased graph $\Omega$ from $H$ as follows. The underlying graph $||\Omega||$ is obtained from $H$ by adding, for every $i \in I$, the set of edges $E_i=\{x_iy\ |\ y \in Y_i\}$.
For every $i \in I$, we declare the graph with edge set $E_i \cup E(H_{i+3})$ to be balanced.
For all distinct $i,j \in I$ and all $f_i \in E_i$ and $f_j \in E_j$, every $\{f_i,f_j\}$-cycle for $H_i \cup H_j \cup H_{i+3}\cup H_{j+3}$ is unbalanced.
The bias of the other cycles in $\Omega$ may be chosen arbitrarily, as long as the theta property is preserved.
In this definition we may replace some of the  $H_i$'s with a single vertex.
If the vertices $x_i$ for $i \in I$ are all distinct,
we call a biased graph constructed in this fashion a {\em tricoloured biased graph} (see Figure~\ref{fig:tric}).

%
%

\section{Linkages and $3$-planar graphs}\label{sec:linkages}
Given four distinct vertices $s_1,s_2,t_1,t_2$ in a graph $G$, we say that two paths $P_1$ and $P_2$ form an {\em $(s_1-t_1,s_2-t_2)$-linkage} if $P_1$ is an $(s_1,t_1)$-path, $P_2$ is an $(s_2,t_2)$-path and $P_1$ and $P_2$ are disjoint.
If $S_1,S_2,T_1,T_2$ are pairwise disjoint sets of vertices of $G$, then we say that $G$ contains an $(S_1-T_1,S_2-T_2)$-linkage if $G$ contains an $(s_1-t_1,s_2-t_2)$-linkage for some $s_1 \in S_1$, $s_2 \in S_2$, $t_1 \in T_1$ and $t_2 \in T_2$.
Independently, Seymour  \cite{Seymour1980293} and Thomassen \cite{thomassen19802} characterised the graphs having no $(s_1-t_1,s_2-t_2)$-linkage.
We state their result using the notation by Yu in~\cite{MR1984647}. We also use other results by Yu; in~\cite{MR1984647} linkages are allowed to be between pairs of vertices that are not necessarily disjoint. We will modify the results in~\cite{MR1984647} according to our setting.

We first need to define 3-planar graphs.
Let $G$ be a graph and let  $\zA=\{A_1,\ldots,A_k\}$ be a (possibly empty) set  of pairwise disjoint subsets of $V(G)$, such that
$N_G(A_i)\cap A_j$ is empty for all $i,j \in [k]$. We define $\proj(G,\zA)$ to be the graph obtained from $G$ by deleting all sets $A_i$ and adding new edges joining each pair of vertices in $N_G(A_i)$.

We say that $(G,\zA)$ is a  {\em $3$-planar} graph if the following hold:
\begin{itemize}
    \item[(a)] $|N_G(A_i)|\leq 3$ for all $i \in [k]$;
    \item[(b)] $\proj(G,\zA)$ is a planar graph and it can be embedded on the plane so that, for each $A_i$ with $|N_G(A_i)|=3$, $N_G(A_i)$ induces a facial triangle.
\end{itemize}

In addition, if $v_1,\ldots,v_n$ are vertices in $G$ such that $v_i \notin A_j$ for all $i \in [n]$ and $j\in [k]$, and $v_1,\ldots,v_n$ occur in this circular order in a face boundary of $\proj(G,\zA)$ (for an embedding as in (b)), then we say that $(G,\zA,(v_1,\ldots,v_n))$ is {\em $3$-planar}. Sometimes we will omit the set $\zA$ and say that $G$ or $(G, (v_1,\ldots,v_n))$ is $3$-planar. The vertices $v_1,\ldots,v_n$ do not need to be all distinct in this definition.
If $(G,\zA,(v_1,\ldots,v_n))$ is 3-planar for some empty set $\zA$, then we say that $(G,(v_1,\ldots,v_n))$ is {\em planar}.

Given two disjoint sets $X$ and $Y$ of vertices in a graph $G$, we say that $(G, (X,Y))$ is {\em $3$-planar} if there is some ordering $(x_1,\ldots,x_k)$ of the vertices in $X$ and some ordering $(y_1,\ldots,y_{\ell})$ of the vertices in $Y$ such that $(G, (x_1,\ldots,x_k,y_1,\ldots,y_{\ell}))$ is $3$-planar. If $X=\{x\}$ then we may abuse notation and write that $(G,(x,Y))$ is $3$-planar.
This definition extends to the case when $x_k=y_1$ and/or $y_{\ell}=x_1$. We also extend this notation in the obvious way to the case when we have more than two sets.

\begin{theorem}[Theorem 2.4 in~\cite{MR1984647}]\label{thm:noLinkage}
Let $G$ be a graph and $s_1,t_1,s_2,t_2$ be distinct vertices in $G$. Then $G$ contains no $(s_1-t_1,s_2-t_2)$-linkage if and only if $(G, (s_1,s_2,t_1,t_2))$ is $3$-planar.
\end{theorem}

Yu's paper~\cite{MR1984647} contains other useful results on linkages that we report next.

Let $(G,\zA)$ be $3$-planar and let $A \in \zA$. We say that $A$ is {\em minimal} if there are no nonempty pairwise disjoint subsets $D_1,\ldots,D_k \subset A$ (where $k \geq 2$) such that $(G,(\zA- A)\cup \{D_1,\ldots,D_k\})$ is $3$-planar. If every $A \in \zA$ is minimal, then we say that $\zA$ is minimal.

\begin{lemma}[Proposition 2.6 in~\cite{MR1984647}]\label{lem:link0}
If $(G,(v_1,\ldots,v_n))$ is 3-planar, then there is a collection $\zA$ of pairwise disjoint subsets of $V(G)-\{v_1,\ldots,v_n\}$ such that
$(G,\zA,(v_1,\ldots,v_n))$ is 3-planar and $\zA$ is minimal.
\end{lemma}

\begin{lemma}[Proposition 3.1 in~\cite{MR1984647}]\label{lem:link1}
Let $(G,\zA)$ be $3$-planar, let $A \in \zA$ with $N_G(A)=\{a_1,a_2,a_3\}$ and let $H=G[A\cup N_G(A)]$.
Suppose that $A$ is minimal. Then the following hold:
\begin{itemize}
    \item[(a)] for any proper subset $X \subset N_G(A)$, $H - X$ is connected;
    \item[(b)] for any $x \in A$, $H$ contains an $(x-a_1,a_2-a_3)$-linkage;
    \item[(c)] for any $x \in A$, $N_G(A)$ is contained in a component of $H - x$.
\end{itemize}
\end{lemma}

\begin{lemma}[Proposition 3.2 in~\cite{MR1984647}]\label{lem:link2}
Let $(G,\zA)$ be $3$-planar, where $\zA$ is minimal.
Let $v_1,v_2 \in V(G)$ be distinct.
Let $u_1 \in V(G)$ and define $u_1^*=u_1$ when $u_1 \in V(\proj(G,\zA))$ and $u_1^*$ to be an arbitrary vertex in $N_G(A_1)$ if $u_1 \in A_1$ for some $A_1 \in \zA$.
Let $u_2 \in V(G)$ and define $u_2^*=u_2$ when $u_2 \in V(\proj(G,\zA))$ and $u_2^*$ to be an arbitrary vertex in $N_G(A_2)$ if $u_2 \in A_2 $ for some $A_2\in \zA$.
Suppose that $v_1,v_2,u_1^*,u_2^*$ are all distinct and
\begin{itemize}
    \item[(i)] $A_1 \neq A_2$ if $A_1$ and $A_2$ are both defined, and
    \item[(ii)] $\proj(G,\zA)$ contains a $(v_1-u_1^*,v_2-u_2^*)$-linkage $L^*$.
\end{itemize}
Then $G$ contains a $(v_1-u_1,v_2-u_2)$-linkage $L$ such that, for any vertex $z$ of $\proj(G,\zA)$, $z$ is contained in a path in $L$ if and only if $z$ is contained in a path in $L^*$.
\end{lemma}

We conclude this section with some results on linkages and 3-planar graphs. The proofs are similar to the proofs of results in~\cite{MR1984647}.

\begin{lemma}\label{lem:link4}
Let $(G,\zA, (v_1,\ldots,v_n))$ be $3$-planar, where $G$ is $2$-connected and $\zA$ is minimal. Then $G$ contains a cycle $C$ such that $v_1,\ldots,v_n$ appear in $C$ in this circular order.
\end{lemma}
\begin{proof}
Let $F$ be a face boundary of $\proj(G,\zA)$ containing $v_1,\ldots,v_n$ (in this circular order). Since $G$ is $2$-connected, so is $\proj(G,\zA)$. Therefore $F$ is a cycle of $\proj(G,\zA)$.
Suppose that $e=a_1a_2 \in E(F)$ is not an edge of $G$. Then $a_1,a_2 \in N_G(A)$ for some $A \in \zA$. If $|N_G(A)|=2$, let $P$ be an $(a_1,a_2)$-path in $G[A]$. Otherwise let $N_G(A)=\{a_1,a_2,a_3\}$ and let $P$ be an $(a_1,a_2)$-path in $G[A- a_3]$.
Substituting $e$ with $P$ in $F$ produces a cycle. Repeating this process for every edge of $F$ that is not in $G$ we obtained the desired cycle.
 \end{proof}

\begin{lemma}\label{lem:link3+}
Let $G$ be a $2$-connected graph. Let $k\geq3$ be an integer and $x,y,v_1,v_2,\ldots,v_k$ be distinct vertices in $G$ such that, for all distinct $i,j\in[k]$, there is no $(x-y,v_i-v_j)$-linkage in $G$. Then $\{x,y\}$ is a 2-vertex-cut of $G$ and each $\{x,y\}$-bridge contains at most one of $v_1,v_2,\cdots,v_k$.
\end{lemma}
\begin{proof}
Since there is no $(x-y,v_1-v_2)$-linkage in $G$, \tref{noLinkage} implies that $(G,(x,v_1,y,v_2))$ is 3-planar.
\lref{link4} implies that $G$ contains a cycle $C$ such that $x,v_1,y,v_2$ appear in $C$ in this circular order.
If there is a $(v_3,C-\{x,y\})$-path in $G$, then $G$ contains either an  $(x-y,v_1-v_3)$-linkage or an $(x-y,v_2-v_3)$-linkage. It follows that $\{x,y\}$ is a 2-vertex-cut of $G$ separating $v_3$ from $v_1$ and $v_2$.
Let $B_1,\ldots,B_n$ be the $\{x,y\}$-bridges of $G$.
Since $B_i-\{x,y\}$ is connected for every $i \in [n]$, if  $B_i$ contains vertices $v_j$ and $v_{\ell}$ (for distinct $j,\ell \in [k]$)
then  $G$ contains an $(x-y,v_j-v_{\ell})$-linkage. It follows that every $\{x,y\}$-bridge contains at most one of $v_1,\ldots,v_k$, and the result holds.
\end{proof}

The following lemma will be used in the proof of Lemma \ref{lem:balanced-sub}.

\begin{lemma}\label{3-ele}
Let $G$ be a $2$-connected graph. Let $X,Y$ be subsets of $V(G)$ with at least three vertices and with $|X\cup Y|\geq 4$. Then for some $x,x'\in X$ and $y,y'\in Y$ with $\{x,x'\}\cap\{y,y'\}=\emptyset$ there is an $(x-x',y-y')$-linkage. 
\end{lemma}

\begin{proof}
Let $x_1,x_2,x_3$ be vertices in $X$ and $y_1,y_2,y_3$ vertices in $Y$ with $\{x_1,x_2\}\cap\{y_1,y_2\}=\emptyset$. Assume that $G$ has no $(x_1-x_2, y_1-y_2)$-linkage. Then Theorem 4.1 and Lemma 4.5 imply that $(G,(x_1,y_1,x_2,y_2))$ is 3-planar and $x_1,y_1,x_2,y_2$ appear in a cycle of $G$ in this order. When $X\cup Y=\{x_1,y_1,x_2,y_2\}$, the lemma is obviously true. So we may assume that $y_3\notin\{x_1,x_2\}$ and $G$ has no  $(x_1-x_2, y_i-y_j)$-linkage for all distinct $i,j\in[3]$. By Lemma 4.6, $\{x_1,x_2\}$ is a 2-vertex-cut of $G$ and each $\{x_1,x_2\}$-bridge contains at most one of $y_1,y_2,y_3$. For $i=1,2,3$, let $B_i$ be the $\{x_1,x_2\}$-bridge containing $y_i$. Without loss of generality we may assume that $x_3 \notin V(B_1) \cup V(B_2)$. Then $G$ contains an $(x_1-x_3, y_1-y_2)$-linkage.
\end{proof}

\begin{lemma}\label{lem:link6}
Let $G$ be a $2$-connected graph, let $X$ and $Y$ be disjoint nonempty sets of vertices of $G$ and let $v_1,v_2 \in V(G) - (X \cup Y)$.
Suppose that for every $x \in X$ and $y \in Y$, $G$ has no $(v_1-v_2,x-y)$-linkage.
Then one of the following occurs.
\begin{itemize}
    \item[(a)] $G$ contains a $2$-separation $(A_1,A_2)$ with $v_1,v_2\in V(A_1)$ and either $X \subset V(A_1)$ and $y \in V(A_2)$ for some $y \in Y$, or $Y \subset V(A_1)$ and $x \in V(A_2)$ for some $x \in X$.
    \item[(b)] $(G,(v_1,X,v_2,Y))$ is $3$-planar.
\end{itemize}
\end{lemma}
\begin{proof}
We prove the result by induction on $|X|+|Y|$. If $|X|=|Y|=1$ the result holds by \tref{noLinkage}.
Now suppose that $|X|\geq 2$ and pick some $x \in X$.
By induction one of the following occurs.
\begin{itemize}
    \item[(1)] $G$ contains a $2$-separation $(A_1,A_2)$ with $v_1,v_2\in V(A_1)$ and either
    \begin{itemize}
        \item[(1.1)] $X-\{x\} \subset V(A_1)$ and $y \in V(A_2)$ for some $y \in Y$, or
        \item[(1.2)] $Y \subset V(A_1)$ and $x' \in V(A_2)$ for some $x' \in X-\{x\}$.
    \end{itemize}
    \item[(2)] $(G,\zA,(v_1,X-\{x\},v_2,Y))$ is $3$-planar for some set $\zA$.
\end{itemize}

If (1.2) occurs, then the same separation $(A_1,A_2)$ satisfies (a) for $X$ and $Y$.
Now suppose that (1.1) occurs. If $x \in V(A_1)$, then again $(A_1,A_2)$ satisfies (a) for $X$ and $Y$. So now assume that $x \in V(A_2)-V(A_1)$.
If no vertex of $Y$ is in $V(A_2)-V(A_1)$, then $(A_1,A_2)$ satisfies (a) for $X$ and $Y$. Now suppose some $y \in Y$ is in $V(A_2)-V(A_1)$.
Then either $G$ contains a $(v_1-v_2,x-y)$-linkage (which is not possible), or $V(A_1) \cap V(A_2) =\{v_1,v_2\}$ and $x$ and $y$ are in different $\{v_1,v_2\}$-bridges (since $G$ is 2-connected). Let $B$ be the $\{v_1,v_2\}$-bridge containing $x$ and define $A_1' = A_1 \cup B$ and $A_2'=A_2 - (B-\{v_1,v_2\})$.
Since $y \in V(A_2')-V(A_1')$,  $(A_1',A_2')$ is a $2$-separation of $G$ satisfying (a).

Now suppose that (2) occurs. We may choose $\zA$ to be minimal.
Set $G'=\proj(G,\zA)$ and let $F$ be the face boundary of $G'$ containing $v_1,v_2,X-\{x\}$ and $Y$.
Let $P_X$ be the $(v_1,v_2)$-path in $F$ containing $X-\{x\}$ and $P_Y$ be the  $(v_1,v_2)$-path in $F$ containing $Y$.

If $x \in P_X$, then (b) holds and we are done. Suppose this is not the case.
Define $Z=\{x\}$ if  $x \in V(G')$ and $Z=N_{G'}(A_x)$ if $x \in A_x$ for some $A_x\in \zA$.
Suppose that there exists a path $Q$ in $G'$ (possibly with no edges) joining some $x^* \in Z$ to $P_Y$ and such that $Q$ and $P_X$ are disjoint. Then $F \cup Q$ contains a $(v_1-v_2,y-x^*)$-linkage for any $y \in Y$. By \lref{link2}, $G$ contains a $(v_1-v_2,y-x)$-linkage,
a contradiction.
Thus every path joining some $x^* \in Z$ to $P_Y$ intersects $P_X$. By the planarity of $G'$, this implies that $G'$ contains a $2$-separation $(H_1,H_2)$ such that $\{v_1,v_2\}\subseteq V(H_1)$ and such that $P_Y$ is contained in $H_1$ and $Z \subseteq V(H_2)$.
Then $(H_1,H_2)$ naturally extends to a $2$-separation in $G$ satisfying (a).
\end{proof}

For the proof of the next lemma we require some new terminology. This terminology will be used throughout the paper.
Let $C$ be a cycle of a $2$-connected graph $G$ and let $x$ be a vertex of $C$.
A vertex $y\in V(G)$ {\em attaches to $C$ at $x$} if there exists an $(x,y)$-path $P$ of $G$ such that $|V(C)\cap V(P)|=1$.
In this case $x$ is an {\em attachment} of $y$ on $C$.
Given a path $P$ in $C$, we say that $y$ {\em only attaches to $P$} if all the attachments of $y$ on $C$ are in $P$.
Since $G$ is 2-connected, $y$ has only one attachment if and only if $y\in V(C)$.

Let $C$ be a cycle and suppose that $x_1,x_2,\ldots,x_n\in V(C)$ occur on $C$ in this cyclic order.
For any two distinct $x_i$ and $x_j$, $C$ contains two $(x_i,x_j)$-paths.
Let $C[x_i,x_{i+1},\ldots, x_j]$ denote the $(x_i,x_j)$-path in $C$ containing $x_i,x_{i+1},\ldots, x_j$ (and not containing $x_{j+1}$ if $i \neq j+1$), where subscripts are modulo $n$. Such path is uniquely determined when $n \geq 3$. Similarly, set
\[\begin{aligned}
C(x_i,x_{i+1},\ldots, x_j]&=C[x_i,x_{i+1},\ldots, x_j]- x_i, \\
C[x_i,x_{i+1},\ldots, x_j)&=C[x_i,x_{i+1},\ldots, x_j]- x_j,\\
C(x_i,x_{i+1},\ldots, x_j)&=C[x_i,x_{i+1},\ldots, x_j]- \{x_i,x_j\}.
\end{aligned}\]

\begin{lemma}\label{lem:link5}
Let $G$ be a $2$-connected graph and let $x_1,\ldots,x_n,y_1,\ldots,y_n$ be distinct vertices of $G$, with $n \geq 2$.
Suppose that $G$ contains no $(x_i-y_i,x_j-y_j)$-linkage, for all distinct $i,j \in [n]$.
Then, up to a reordering of $[n]$ and swapping some $x_i$ with $y_i$, $(G,(x_1,\ldots,x_n,y_1,\ldots,y_n))$ is 3-planar.
\end{lemma}

\begin{proof}
We prove this by induction on $n$. When $n=2$, the lemma follows from \tref{noLinkage}.
So we may assume that $n\geq3$ and the result holds for $n-1$. Therefore,
there is a collection $\zA$ of pairwise disjoint subsets of $V(G)-\{x_1,\ldots,x_{n-1}, y_1,\ldots, y_{n-1}\}$ such that
$(G,\zA,(x_1,\ldots,x_{n-1}, y_1,\ldots, y_{n-1}))$ is 3-planar. We may choose $\zA$ to be minimal.
Since $\zA$ is minimal, if $x_n$ and $y_n$ both belong to a same set $A \in \zA$, then, by \lref{link1}(a), $G$ contains an $(x_1-y_1, x_n-y_n)$-linkage.
Thus we may assume that $x_n$ and $y_n$ do not belong to a same set $A \in \zA$.

Let $H=\proj(G,\zA)$. Since $G$ is $2$-connected, so is $H$.
Let $C$ be the face boundary of $H$ containing $x_1,\ldots,x_{n-1}, y_1,\ldots, y_{n-1}$ (in this circular order).
Define $X_n=\{x_n\}$ if $x_n \in V(H)$ and $X_n=N_{G'}(A_{x_n})$ if $x_n \in A_{x_n}$ for some $A_{x_n}\in \zA$.
Similarly, define $Y_n=\{y_n\}$ if $y_n \in V(H)$ and $Y_n=N_{G'}(A_{y_n})$ if $y_n \in A_{y_n}$ for some $A_{y_n}\in \zA$.

Suppose that $H$ contains an $(x^*-y^*,x_i-y_i)$-linkage for some $x^* \in X_n$, $y^* \in Y_n$ and $i \in [n-1]$.
Then by \lref{link2}, $G$ contains an $(x_n-y_n,x_i-y_i)$-linkage, a contradiction.

\begin{claim}
We may assume that every vertex in $X_n$ only attaches to $C[x_{n-1},y_1]$ and every vertex in $Y_n$ only attaches to $C[y_{n-1},x_1]$.
\end{claim}
\begin{cproof}
By possibly swapping some $x_i$ with $y_i$ and changing subscripts appropriately, we may assume that some $x^* \in X_n$ attaches to $C(x_{n-1},y_1]$.

\textbf{Case 1:} $x^*=y_1$. Then every vertex in $Y_n$ attaches only to $C[y_{n-1},x_1,x_2]$.
Since $G$ is 2-connected and $y_n \notin \{x_1,\ldots,x_n,y_1,\ldots,y_{n-1}\}$,
for any $z \in \{x_1,\ldots,x_n,y_1,\ldots,y_{n-1}$\} the set $Y_n -\{z\}$ is non-empty.
Thus, if there exists a vertex $x' \in X_n$ which attaches to a vertex not in $C[x_{n-1},y_1,y_2]$, then $H$ contains either an
$(x'-y^*,x_2-y_2)$-linkage or an $(x'-y^*,x_{n-1}-y_{n-1})$-linkage for some $y^* \in Y_n$, a contradiction.
Therefore every vertex in $X_n$ attaches only to $C[x_{n-1},y_1,y_2]$.
Moreover, if there are vertices $x',x''\in X_n$ such that $x'$ attaches to $C[x_{n-1},y_1)$ and $x''$ attaches to $C(y_1,y_2]$, then
$H$ contains either an $(x'-y^*,x_1-y_1)$-linkage or an $(x''-y^*,x_1-y_1)$-linkage for some $y^* \in Y_n$, again a contradiction.
Thus we may assume that all vertices in $X_n$ attach only to $C[x_{n-1},y_1]$.
This, together with the fact that $X_n$ contains a vertex other than $y_1$, forces all the vertices in $Y_n$ to attach only to $C[y_{n-1},x_1]$, and the claim holds.

\textbf{Case 2:} $x^*\neq y_1$.
Arbitrarily choose $y^*\in Y_n$; then $y^*$ only attaches to $C[y_{n-1},x_1]$ (otherwise $G$ contains either an  $(x^*-y^*,x_1-y_1)$-linkage or an $(x^*-y^*,x_{n-1}-y_{n-1})$-linkage, a contradiction). If some $y^* \in Y_n$ attaches to $C(y_{n-1},x_1)$, then the symmetric argument shows that every vertex in $X_n$ attaches only to $C[x_{n-1},y_1]$. Otherwise $x_1 \in Y_n$, and we conclude by the symmetric argument to the one in Case 1.
\end{cproof}

If $x_n,y_n \in V(C)$, then we are done by the claim.
Now suppose that $x_n \notin V(C)$ and $y_n \in V(C)$. Let $x^* \in X_n$; since $x^*$ only attaches to $C[x_{n-1},y_1]$, by the planarity of $H$ there exists a $2$-vertex-cut $Z \subseteq V(C[x_{n-1},y_1])$ such that $x^*$ and $C(y_1,\ldots,y_{n},x_1,\ldots,x_{n-1})$ are in different $Z$-bridges. Since the vertices in $X_n$ are pairwise adjacent, $x$ and the other vertices in $X_n$ are in the same $Z$-bridge $B^*$.
The $2$-vertex-cut $Z$ is also a $2$-vertex-cut of $G$; let $B$ be the $Z$-bridge in $G$ corresponding to $B^*$.
If $B-Z-\{x_n\} \neq\emptyset$ set
$$\zA'=(\zA-\{A\in\zA |\ A\subseteq V(B)\})\cup\{B-Z-\{x_n\}\},$$
otherwise set $\zA '=\zA$. Since $|N_G(A)|\leq3$ for every $A\in \zA'$ and $Z \subseteq V(C[x_{n-1},y_1])$, we have that $(G,\zA',(x_1,\ldots,x_n,y_1,\ldots,y_n))$ is 3-planar.

Now suppose that both $x_n$ and $y_n$ are not in $C$. Then we may apply a similar argument to the one above, once for $x_n$ and once for $y_n$ and obtain a new set $\zA'$ such that $(G,\zA',(x_1,\ldots,x_n,y_1,\ldots,y_n))$ is 3-planar.
\end{proof}

\section{Small separations}\label{sec:separations}

In this section we will show that we can reduce our problem to the case where the tangled biased graph is $4$-connected.
To do so we will show that if $\Omega$ is not $4$-connected then either $\Omega$ is a generalised wheel, or we can obtain $\Omega$ as a $1$-, $2$- or $3$-sum of a balanced graph and a tangled biased graph.
We first need to define summing operations on biased graphs.

Let $\Omega_1=(G_1,\bal_1)$ and $\Omega_2=(G_2,\bal_2)$ be two biased graph, where $\Omega_2$ is balanced.
Suppose that both $\Omega_1$ and $\Omega_2$ contain a balanced $K_t$ subgraph, for some $t\in[3]$ and $|V(\Omega_1)|,|V(\Omega_2)| > t$. Then the graph $G_1 \oplus_t G_2$ is the graph obtained from $G_1$ and $G_2$ by identifying the common $K_t$ and deleting the edges of $K_t$.
We define $\bal=\bal_1\oplus_t \bal_2$ as follows.
If $t=1$, then $\bal$ is just the union of $\bal_1$ and $\bal_2$.
If $t=2$, let $e$ be the edge in the $K_t$.
Then $\bal=\{C\in \bal_1 \cup \bal_2 : e \notin C\} \cup \{(C_1 \cup C_2)\bs e : e\in C_1\in \bal_1, e\in C_2 \in \bal_2\}$.
If $t=3$, let $F$ be the edge set of the $K_t$. Then
$\bal$ is the union of the set $\{C\in \bal_1 \cup \bal_2 : C \cap F =\emptyset\}$ and, for every $e\in F$, the sets of the form  $\{(C_1 \cup C_2)\bs e: C_1\in \bal_1, C_2 \in \bal_2, C_1\cap F= C_2\cap F=\{e\}\}$.
Finally we define $\Omega_1 \oplus_t \Omega_2$ as $(G_1\oplus_t G_2,\bal_1 \oplus_t \bal_2)$. It is easy to check that, since $\Omega_2$ and the $K_t$ are balanced, $\Omega_1 \oplus_t \Omega_2$ is a biased graph.
We say that $\Omega_1 \oplus_t \Omega_2$ is the {\em $t$-sum} of $\Omega_1$ and $\Omega_2$ on $V(K_t)$.

\begin{lemma}\label{lem:1sep}
Let $\Omega=(G,\bal)$ be a tangled biased graph and suppose that $x$ is a cutvertex of $G$. Then $\Omega$ is a $1$-sum of a tangled biased graph and a balanced graph.
\end{lemma}

\begin{proof}
Let $\Omega_1,\ldots,\Omega_k$ be the bridges of $\{x\}$. Suppose that two of these bridges are unbalanced. Since two bridges have only the vertex $x$ in common, all unbalanced cycles of $\Omega$ contain $x$.
This is not possible, since $\Omega$ has no blocking vertex. Hence we may assume that $\Omega_2,\ldots,\Omega_k$ are balanced,
so $\Omega$ is the $1$-sum of $\Omega_1$ and the balanced biased graph $\Omega_2\cup\cdots\cup\Omega_k$.
Clearly $\Omega_1$ has no two disjoint unbalanced cycles. Moreover, if $\Omega_1$ is balanced, or contains a blocking vertex, then so does $\Omega$. It follows that $\Omega_1$ is tangled, and $\Omega$ is a $1$-sum of a tangled biased graph and a balanced graph.
\end{proof}

\begin{lemma}\label{lem:2sep}
Let $\Omega=(G,\bal)$ be a tangled biased graph and suppose that $\{x_1,x_2\}$ is a $2$-vertex-cut of $G$. Then $\Omega$ is a $2$-sum of a tangled biased graph and a balanced graph.
\end{lemma}
\begin{proof}
Let $\Omega_1,\ldots,\Omega_k$ be the bridges of $\{x_1,x_2\}$.
\begin{claim}\label{cl:2specl1}
Exactly one of $\Omega_1,\ldots,\Omega_k$ is unbalanced.
\end{claim}

\begin{cproof}
If all the $\{x_1,x_2\}$-bridges are balanced, then either $\Omega$
is balanced or all unbalanced cycles of $\Omega$ use both $x_1$ and
$x_2$ (and $\Omega$ has a blocking vertex). Since $\Omega$ is
tangled, this is not the case. Thus we may assume that $\Omega_1$ is
unbalanced. Now assume by way of contradiction that $\Omega_2$ is
also unbalanced. Then every unbalanced cycle of $\Omega_1$ and
$\Omega_2$ contains $x_1$ or $x_2$.

We claim that every unbalanced cycle of $\Omega_1$ uses both $x_1$
and $x_2$. Assume to the contrary that $\Omega_1$ contains an
unbalanced cycle $C_1$ with $V(C_1)\cap\{x_1,x_2\}=\{x_1\}$. Then
every unbalanced cycle contained in bridges other than $\Omega_1$
uses $x_1$. Since $x_1$ is not a blocking vertex of $\Omega$, there
exists an unbalanced cycle $C_2$ not using $x_1$. Thus, $C_2$ must
be contained in $\Omega_1$ and uses $x_2$. Let $C$ be any unbalanced
cycle contained in $\Omega_2$. Since $C$ must intersect both $C_1$
and $C_2$ and $V(C_1)\cap\{x_1,x_2\}=\{x_1\},
V(C_2)\cap\{x_1,x_2\}=\{x_2\}$, we have that $C$ uses both $x_1$ and
$x_2$. Let $P$ be a path in $\Omega_2- \{x_1,x_2\}$ connecting the
two components of $C- \{x_1,x_2\}$. By the definition of bridge,
such $P$ obviously exists. The two cycles in $C\cup P$ other than
$C$ are balanced, since each one does not intersect one of $C_1$ or
$C_2$. Then $C\cup P$ is a theta subgraph with exactly two balanced
cycles, a contradiction. Hence, every unbalanced cycle of $\Omega_1$
uses both $x_1$ and $x_2$.

Let $C$ be an unbalanced cycle of $\Omega_1$. Let $P$ be a path in $\Omega_1- \{x_1,x_2\}$ connecting the two components of $C- \{x_1,x_2\}$. By the above claim the two cycles in $C\cup P$ other than $C$ are balanced. Thus, $C\cup P$ is a theta subgraph with exactly two balanced cycles, a contradiction.
\end{cproof}

By \clref{2specl1} we may assume that only $\Omega_1$ is
unbalanced. Let $G_1$ be the graph obtained from $||\Omega_1||\cup
\cdots \cup ||\Omega_{k-1}||$ by adding a new edge $e$ between $x_1$
and $x_2$, and $G_2$ be obtained from $||\Omega_k||$ by adding a new
edge $e$ between $x_1$ and $x_2$. Set

\[\begin{aligned}
\bal_1=&\{C\ |\ C\ \text{is\ a\ balanced\ cycle\ of}\ \Omega_1 \cup \cdots \cup \Omega_{k-1}\}\cup\{P\cup e\ |\ P\ \text{is\ an}\ (x_1,x_2) \text{-path\ of}\ G_1-e,\\
&\text{and}\ G_2-e\ \text{has\ an}\ (x_1,x_2) \text{-path}\ Q\ \text{such\ that}\ P\cup Q\ \text{is\ balanced\ in}\ \Omega\},\\
\bal_2=&\{C\ |\ C\ \text{is\ a\ cycle\ of}\ G_2\}.
\end{aligned}\]

Suppose that $P$ is an $(x_1,x_2)$-path contained in $||\Omega_i||$ for some $i\in[k-1]$ and suppose that $P \cup Q$ is a balanced cycle for some $(x_1,x_2)$-path $Q$ in $||\Omega_{k}||$.
Let $Q'$ be any other $(x_1,x_2)$-path in $||\Omega_k||$. Then $Q$ may be obtained from $Q'$ by rerouting on cycles in $||\Omega_k||$. Since $\Omega_k$ is balanced, all such cycles are balanced. Therefore $P\cup Q$ and $P \cup Q'$ have the same bias.
From this, it is easy to verify that with this definition, $\Omega=(G_1,\bal_1)\oplus_2(G_2,\bal_2)$.
By definition, $(G_2,\bal_2)$ is balanced.
To conclude the proof of the lemma it remains to show that $(G_1,\bal_1)$ is tangled.

If $(G_1,\bal_1)$ contains two disjoint unbalanced cycles $C_1$ and $C_2$, then either these are disjoint unbalanced cycles of $\Omega$ (which is not possible) or one of $C_1$ and $C_2$ (say $C_1$) contains $e$. By definition of $\bal_1$, for every $(x_1,x_2)$-path $Q$ in $G_2-e$ we have that $C_1'=C_1-e\cup Q$ is an unbalanced cycle of $\Omega$. Thus $C_1'$ and $C_2$ are disjoint unbalanced cycles in $\Omega$, a contradiction. We deduce that $(G_1,\bal_1)$ has no two disjoint unbalanced cycles.

Now suppose that there is a vertex $v$ in $G_1$ such that $(G_1,\bal_1)- v$ is balanced. Since $\Omega$ has no blocking vertex,
there exists an unbalanced cycle $C$ of $\Omega$ avoiding $v$. Such cycle $C$ cannot be contained in $\Omega_1\cup\cdots\cup\Omega_{k-1}$, since $v$ is a blocking vertex for this biased graph; moreover, $C$ cannot be contained in $\Omega_k$, since this biased graph is balanced.
Therefore $C=P\cup Q$, where $P$ is an $(x_1,x_2)$-path in $||\Omega_i||$ for some $i\in[k-1]$ and $Q$ is an $(x_1,x_2)$-path in $||\Omega_k||$. By definition of $\bal_1$, it follows that $P\cup \{e\}$ is an unbalanced cycle of $(G_1,\bal_1)$ avoiding $v$, a contradiction. This also shows that $(G_1,\bal_1)$ is not balanced (since in this case we may pick $v$ to be any vertex). Therefore $(G_1,\bal_1)$ is tangled and the result holds.
\end{proof}

\begin{lemma}\label{lem:3sep}
Let $\Omega=(G,\bal)$ be a tangled biased graph. Suppose that $G$ is the union of two connected graphs $G_1$ and $G_2$, where $V(G_1) \cap V(G_2)=\{x_1,x_2,x_3\}$, $|V(G_1)|,|V(G_2)|\geq 4$ and $\Omega[G_2]$ is balanced.
Then $\Omega$ is a $3$-sum of a tangled biased graph and a balanced graph. 
\end{lemma}
\begin{proof}
Let $\Omega_1,\ldots,\Omega_k$ be the bridges of $\{x_1,x_2,x_3\}$ in $\Omega$. One of these bridges, say $\Omega_k$, is balanced. We proceed similarly to the proof of \lref{2sep} to construct two biased graphs $(G_1,\bal_1)$ and $(G_2,\bal_2)$ such that $(G_1,\bal_1)$ is tangled, $(G_2,\bal_2)$ is balanced and $\Omega=(G_1,\bal_1)\oplus_3 (G_2,\bal_2)$.
Let $G_1$ be the graph obtained from $||\Omega_1||\cup \cdots\cup ||\Omega_{k-1}||$ by adding three new edges $e_{12},e_{23},e_{13}$, where the edge $e_{ij}$ is between $x_i$ and $x_j$.
Let $G_2$ be the graph obtained from $||\Omega_k||$ by adding three new edges $e_{12},e_{23},e_{13}$, where the edge $e_{ij}$ is between $x_i$ and $x_j$.
Set $\bal_2$ to be the set of all cycles in $G_2$. It remains to define $\bal_1$.
The set $\bal_1$ contains all balanced cycles of $\Omega_1 \cup \cdots\cup \Omega_{k-1}$ and the cycle $\{e_{12},e_{13},e_{23}\}$ plus the cycles using $e_{12},e_{13},e_{23}$ which we will discuss next.
Let $Q$ be an $(x_1,x_2)$-path in $||\Omega_k||$.
For every $(x_1,x_2)$-path $P$ in $G_1\bs \{e_{12},e_{13},e_{23}\}$, we add $P\cup \{e_{12}\}$ to $\bal_1$ if and only if $P \cup Q$ is balanced; we add $P\cup \{e_{13},e_{23}\}$ to $\bal_1$ if and only if $P \cup Q$ is balanced and $P$ does not use vertex $x_3$. We define the bias of the other cycles using the three new edges similarly. Since we declared the cycle $\{e_{12},e_{13},e_{23}\}$ to be balanced, it can be checked that $(G_1,\bal_1)$ is indeed a biased graph and that it is tangled.
We leave it to the reader to check that $\Omega=(G_1,\bal_1)\oplus_3 (G_2,\bal_2)$.
\end{proof}

\begin{lemma}\label{lem:4conn}
Let $\Omega=(G,\bal)$ be a simple tangled biased graph. If $\Omega$ is not $4$-connected, then either
\begin{itemize}
    \item $\Omega$ is a $t$-sum of a tangled biased graph and a balanced graph, for some $t\in [3]$, or
    \item $\Omega$ is a generalized wheel.
\end{itemize}
\end{lemma}
\begin{proof}
Suppose that $\Omega$ is not $4$-connected. By \lref{1sep}, \lref{2sep} and \lref{3sep}, we may assume that $\Omega$ is $3$-connected and contains a $3$-vertex-cut $X=\{x_1,x_2,x_3\}$ and all the bridges of $X$ are unbalanced.
Let $\Omega_1,\ldots,\Omega_k$ be the bridges of $X$.

\begin{claim}\label{cl:4connCl1}
For every bridge $\Omega_i$ there exists an unbalanced cycle in $\Omega_i$ using exactly one of $x_1,x_2,x_3$.
\end{claim}
\begin{cproof}
Since $\Omega$ has at least two unbalanced $X$-bridges, each unbalanced cycle must intersect $X$.
Suppose that $C$ is an unbalanced cycle contained in $\Omega_i$ using at least two vertices in $X$.
Let $P$ be a minimal path in $\Omega_i- X$ joining two components of $C- X$. The subgraph $C\cup P$ is a theta subgraph, and by the choice of $P$, no vertex of $X$ is in $P$.
Let $C_1$ and $C_2$ be the cycles in $C\cup P$ other than $C$. By the theta property at least one of $C_1$ and $C_2$ is unbalanced. Therefore one of $C_1$ or $C_2$ is the required cycle, unless one of them, say $C_1$, is unbalanced and contains two vertices in $X$, while $C_2$ is balanced. Thus in this case $C$ contains all of the vertices in $X$. However, $C_1$ is an unbalanced cycle in $\Omega_i$ not using all vertices in $X$, so we may repeat the same argument with $C_1$ in place of $C$, and conclude that there exists an unbalanced cycle $C$ using exactly one of $x_1,x_2,x_3$.
\end{cproof}

The following is an immediate consequence of \clref{4connCl1}.

\begin{claim}\label{cl:4connCl1+}
For any 3-vertex-cut $X'$ of $G$, each $X'$-bridge has an unbalanced cycle using exactly one element of $X'$.
\end{claim}

By \clref{4connCl1}, there exist unbalanced cycles $C_1$ and $C_2$ contained in $\Omega_1$ and $\Omega_2$ respectively, each using exactly one vertex in $X$. Thus $C_1$ and $C_2$ use the same vertex, say $x_1$, of $X$.
Since $x_1$ is not a blocking vertex of $\Omega$, there exists an unbalanced cycle $C$ avoiding $x_1$. Since $C$ intersects both $C_1$ and $C_2$, we have $C=P_1\cup P_2$, where $P_1$ is an $(x_2,x_3)$-path of $\Omega_1$ avoiding $x_1$ and
$P_2$ is an $(x_2,x_3)$-path of $\Omega_2$ avoiding $x_1$.

Now suppose that there is a third $X$-bridge $\Omega_3$. Then $\Omega_3$ contains an unbalanced cycle $C_3$ using at most one vertex in $X$. However, such cycle is disjoint from either $C_1$ or $C$. It follows that the only bridges of $X$ are $\Omega_1$ and $\Omega_2$. Since $X$ was chosen arbitrarily, we have

\begin{claim}\label{cl:4connCl1++}
For any 3-vertex-cut $X'$ of $G$, there are exactly two $X'$-bridges.
\end{claim}

Because of $C_1$, the vertex $x_1$ is a blocking vertex of $\Omega_2$. Similarly, $x_1$ is a blocking vertex of $\Omega_1$.
Let $Q_1$ be any $(x_2,x_3)$-path in $\Omega_1- x_1$. All the cycles in $\Omega_1 - x_1$ are balanced, therefore $Q_1\cup P_2$ has the same bias as $C=P_1\cup P_2$ (i.e. $Q_1\cup P_2$ is unbalanced). The same argument holds if we replace $P_2$ with some other $(x_2,x_3)$-path in $\Omega_2-x_1$. It follows that every $(x_2,x_3)$-path in $\Omega_1$ intersects $C_1$ and
every $(x_2,x_3)$-path in $\Omega_2$ intersects $C_2$.
Since $C_1$ and $C_2$ were arbitrary unbalanced cycles (avoiding $x_2$ and $x_3$), we have the following result.

\begin{claim}\label{cl:4connCl2}
Every unbalanced cycle in $\Omega_i$ intersects every $(x_2,x_3)$-path in $\Omega_i$.
\end{claim}

Next we focus on the structure of $\Omega_1$.
Let $H_1,\ldots,H_n$ be the blocks of $\Omega_1-x_1$.
Suppose that  $n \geq 2$ and $H_i$ is a leaf block that contains neither $x_2$ nor $x_3$ in its interior. Then $E(H_i)$ together with the edges joining $H_i$ to $x_1$ forms a $2$-separation of $G$, a contradiction. Therefore the tree of blocks of $\Omega_1-x_1$ is a path (possibly with only one vertex) and its ends each contain one of $x_2$ and $x_3$ in the interior. We relabel the blocks so that (for every $i \in [n-1]$) $H_i$ and $H_{i+1}$ share a vertex, and $x_2 \in V(H_1)-V(H_2)$ and $x_3 \in V(H_n)-V(H_{n-1})$ (or, if $n=1$, we simply have $x_2,x_3 \in V(H_1)$). For every $i \in [n-1]$, let $z_i$ be the vertex shared by $H_i$ and $H_{i+1}$. Set $z_0:=x_2$ and $z_n:=x_3$.

Recall that $x_1$ is a blocking vertex of $\Omega_1$.
Let $U_1,\ldots,U_{\ell}$ be the parts of the standard partition of $\delta_{\Omega_1}(x_1)$.
For every $i \in [\ell]$, let $Y_i$ be the set of vertices adjacent to $x_1$ with an edge in $U_i$.
Arbitrarily choose $H_t$ such that $H_t$ is not a single edge.
Since $\Omega$ is simple and there are no unbalanced cycles in $H_t$, the block $H_t$ contains at least one vertex other than $z_{t-1}$ and $z_t$.
Therefore $\{z_{t-1},z_t,x_1\}$ is a $3$-vertex-cut of $G$; since $H_t$ contains no unbalanced cycles, \clref{4connCl1+} implies that at least two of $Y_1,\ldots,Y_{\ell}$ intersects $H_t - \{z_{t-1},z_t\}$. We claim that exactly two of the sets $Y_1,\ldots,Y_{\ell}$ intersect $H_t - \{z_{t-1},z_t\}$. Assume to the contrary that there are three distinct $Y_i,Y_j,Y_k$ intersecting $H_t - \{z_{t-1},z_t\}$. Arbitrarily choose $y_i \in Y_i\cap(V(H_t) - \{z_{t-1},z_t\}), y_j \in Y_j\cap(V(H_t) - \{z_{t-1},z_t\})$, and $y_k \in Y_k\cap(V(H_t) - \{z_{t-1},z_t\})$. 
If, say, $y_i=y_j$, then (since $H_t$ is a block) there is a $(z_{t-1},z_t)$-path in $H_t$ disjoint from the unbalanced cycle made of the two parallel $x_1y_i$ edges (one from $U_i$ and one from $U_j$), contradicting \clref{4connCl2}. Thus the vertices $y_i,y_j,y_k$ are all distinct.
By \clref{4connCl2} for any $s,p\in\{i,j,k\}$ every $(y_s,y_p)$-path in $H_t$ intersects every $(z_{t-1},z_t)$-path in $H_t$. That is, $H_t$ contains no $(z_{t-1}-z_t,y_s-y_p)$-linkage. Therefore, by \lref{link3+} $\{z_{t-1},z_t\}$ is a 2-vertex-cut of $H_t$ with at least three $\{z_{t-1},z_t\}$-bridges; consequently, the 3-vertex-cut $\{x_1,z_{t-1},z_t\}$ of $G$ has at least three $\{x_1,z_{t-1},z_t\}$-bridges, a contradiction to \clref{4connCl1++}.

Assume that $Y_i$ and $Y_j$ intersect $H_t - \{z_{t-1},z_t\}$. Set
$$Y_i'=Y_i \cap (V(H_t) - \{z_{t-1},z_t\})\ \textrm{and} \ Y_j'=Y_j \cap (V(H_t) - \{z_{t-1},z_t\}).$$
\lref{link6} implies that either
\begin{itemize}
    \item[(a)] $H_t$ contains a $2$-separation $(A_1,A_2)$ with $z_{t-1},z_t\in V(A_1)$ and either $Y_i' \subset V(A_1)$ or $Y_j' \subset V(A_1)$, or
    \item[(b)] $(H_t,(z_{t-1},Y_i',z_t,Y_j'))$ is $3$-planar.
\end{itemize}
To conclude our proof it remains to show that case (a) does not occur.
Suppose a separation as in (a) exists, with  $Y_i' \subset V(A_1)$. Set $V(A_1)\cap V(A_2)=\{p,q\}$. Let $H'$ be the subgraph of $H_t$ induced by the edges in $A_2$ together with the edges joining $x_1$ to $V(A_2)-\{p,q\}$. Since $Y_i' \subset V(A_1)$, the $\{x_1,p,q\}$-bridge $H'$ of $G$ contains no unbalanced cycle, a contradiction to \clref{4connCl1+}.
\end{proof}

\section{Finding a balanced subgraph}\label{sec:balanced_sub}

\begin{lemma}\label{lem:small-balanced-sub}
Suppose that $\Omega$ is a simple 4-connected tangled biased graph with at least six vertices. 
Then one of the following holds:
\begin{itemize}
	\item $\Omega$ has a $2$-connected balanced subgraph with at least four vertices, or
	\item $\Omega$ is a tricoloured graph.
\end{itemize}
\end{lemma}

\begin{proof}
Let $\Omega'$ be a $2$-connected balanced subgraph of $\Omega$ where $|V(\Omega')|$ is maximal.
By Theorem 1.1, if $\Omega$ has no balanced cycle then either it has at most 5 vertices or it is not 4-connected.
Thus $\Omega$ has at least one balanced cycle, so $|V(\Omega')|\geq3$. Assume that $|V(\Omega')|=3$. Then 
\begin{itemize}
\item[(1.1)] each cycle in $\Omega$ of length at least four is unbalanced, and
\item[(1.2)] each cycle in $\Omega$ sharing some edges with a balanced triangle is unbalanced (by the theta property).
\end{itemize}

{\bf Case 1:} $|V(\Omega)|=6$. 

Set $V(\Omega)=\{u_1,u_2,\ldots,u_6\}$. Let $K_6$ be a complete graph defined on $V(\Omega)$. Since $\Omega$ is 4-connected, by symmetry we may assume that $K_6\del\{u_1u_2, u_3u_4, u_5u_6\}$ is a subgraph of $\Omega$. 
Since $u_1u_3u_5u_1$ and $u_2u_4u_6u_2$ are disjoint cycles, by symmetry we may assume that the former is balanced. Then every other cycle using an edge in this triangle is unbalanced by (1.2).
In particular, $u_2u_3u_5u_2$ is unbalanced. Moreover, since $u_2u_3u_5u_2$ and $u_1u_4u_6u_1$ are disjoint cycles, $u_1u_4u_6u_1$ is balanced; so $u_2u_4u_6u_2$ is unbalanced by (1.2). Using this same strategy several times one can check that there are exactly four balanced cycles in $K_6\del\{u_1u_2, u_3u_4, u_5u_6\}$, which are $u_1u_3u_5u_1$, $u_1u_4u_6u_1$, $u_2u_3u_6u_2$ and $u_2u_4u_5u_2$.

Now assume that $u_1$ and $u_2$ are adjacent in $\Omega$. By (1.2) (applied to the known balanced triangles) all triangles in $K_6\del\{u_3u_4, u_5u_6\}$ containing the edge $u_1u_2$ are unbalanced. Similar results hold when $u_3u_4$ or $u_5u_6$ are in $E(\Omega)$. Since $\Omega$ has no two disjoint unbalanced cycles, at most one edge in $\{u_1u_2, u_3u_4, u_5u_6\}$ is in $\Omega$; so we may assume that either $\Omega$ is $K_6\del\{u_1u_2, u_3u_4, u_5u_6\}$ or it is this graph together with the edge $u_1u_2$. This graph is depicted on the left of Figure~\ref{fig:sporadic}. 
One can see that this graph is a tricoloured graph with (in the definition of tricoloured graph): $I=\{1,2,3\}$, $V(H_1)=\{u_1,u_3,u_5\}$, $V(H_2)=\{u_2,u_3,u_6\}$, $V(H_3)=\{u_2\}$, $V(H_4)=\{u_2,u_4\}$, $V(H_5)=\{u_1,u_4\}$, $V(H_6)=\{u_1\}$ and 
$Y_1=\{u_2,u_4\}$, $Y_2=\{u_1,u_4\}$ and $Y_3$ is either empty or equal to $\{u_1\}$ (depending on whether the edge $u_1u_2$ is present or not).

{\bf Case 2:} $|V(\Omega)|\geq7$. 

Let $C$ be an unbalanced cycle of $\Omega$ with $|V(C)|$ as small as possible. Then $C$ is an induced subgraph of $\Omega$. Moreover, $\Omega-V(C)$ is balanced and therefore all its blocks are isomorphic to $K_1$, $K_2$ or $K_3$. Note that $C$ cannot be a loop (otherwise $\Omega$ would have a blocking vertex) and if $C$ has exactly two vertices $v_1$ and $v_2$ then $\Omega-\{v_1,v_2\}$ is a 2-connected balanced graph with at least four vertices, against our assumptions. 
Therefore $|V(C)|\geq3$. 

\begin{claim}
$|V(C)|=3$.
\end{claim}

\begin{cproof}
Assume that there are edges $f_1,f_2,f_3$ joining some vertex of $\Omega-V(C)$ and $V(C)$. 
By (1.1) and (1.2) at most one cycle contained in $C\cup\{f_1,f_2,f_3\}$ is balanced. Then by the choice of $C$ we have $|V(C)|=3$. So we may assume that each vertex in $\Omega-V(C)$ has at most two neighbours in $C$, implying that all leaf blocks of $\Omega-V(C)$ are isomorphic to $K_3$. 

Let $B$ be a leaf block of $\Omega-V(C)$, and $p_1,p_2$ be two vertices in $B$ not contained in other blocks of $\Omega-V(C)$. If some vertex in $C$ has two neighbours in $B$, then by (1.2) and the choice of $C$ we have $|V(C)|=3$. Hence, 
\begin{itemize}
\item[(1.3)] each vertex in $C$ has at most one neighbour in $B$ when $|V(C)|\geq4$. 
\end{itemize}
Let $y_i$ and $z_i$ be two distinct neighbours of $p_i$ in $C$. By (1.3) we may assume that $\{y_1,z_1\}\cap\{y_2,z_2\}=\emptyset$. Without loss of generality we may further assume that there is a $(y_1,z_2)$-path $P$ on $C$ containing neither $z_1$ nor $y_2$. Since $P\cup\{p_1y_1, p_2z_2, p_1p_2\}$ is an unbalanced cycle of length less than or equal to $|V(C)|$ by (1.1), by the choice of $C$ we have $V(C)=V(P)\cup\{y_2,z_1\}$. Moreover, since $p_1p_2y_2z_1p_1$ is an unbalanced cycle by (1.1), we have $V(C)=\{y_1,z_1, y_2,z_2\}$.

We rename the vertices of $C$ so that $C=v_1v_2v_3v_4v_1$. Since $|V(C)|=4$, all triangles in $\Omega$ are balanced. By (1.3), there is a partition $(\{v_i,v_j\}, \{v_s,v_t\})$ of $V(C)$ such that $\{v_i,v_j\}=N_{\Omega}(p_1)\cap V(C), \{v_s,v_t\}=N_{\Omega}(p_2)\cap V(C)$. We say that the partition of $V(C)$ is {\sl determined by} $B$. 
When the partitions of $V(C)$ determined by two leaf blocks $B$ and $B'$ are the same, $\Omega$ has two disjoint 4-cycles (which are unbalanced by our assumptions). So every two leaf blocks determine different partitions of $V(C)$. 

Let $w$ be the vertex in $V(B)-\{p_1,p_2\}$; by (1.3) $w$ has no neighbours in $C$, so the component of $\Omega-V(C)$ containing $B$ has another leaf block $B'$. Let $q_1,q_2$ be the vertices in $B'$ not contained in other blocks of $\Omega-V(C)$.  

When $\{v_1,v_3\}=N_{\Omega}(p_1)\cap V(C), \{v_2,v_4\}=N_{\Omega}(p_2)\cap V(C)$, $\{v_1,v_2\}=N_{\Omega}(q_1)\cap V(C)$, and $\{v_3,v_4\}=N_{\Omega}(q_2)\cap V(C)$, we have that $v_2v_3v_4p_2v_2$ and $v_1p_1Pq_1v_1$ are disjoint unbalanced cycles, where $P$ is a $(p_1,q_1)$-path in $\Omega-V(C)$ avoiding $p_2$. Hence, by symmetry we may assume that neither $B$ nor $B'$ determines the partition $(\{v_1,v_3\}, \{v_2,v_4\})$ of $V(C)$. Then $\Omega-V(C)$ has exactly two leaf blocks for otherwise two of them determine the same partition of $V(C)$. When $\Omega-V(C)\neq B\cup B'$, by (1.2) and the fact that each vertex in $\Omega$ has degree at least four, $\Omega$ has two disjoint unbalanced cycles. So $\Omega-V(C)=B\cup B'$. Since $w$ has no neighbours in $C$, $\Omega$ is isomorphic to the graph pictured in the middle of Figure~\ref{fig:sporadic}, where all triangles are balanced. From the picture one can see that this is a tricoloured graph.
\end{cproof}

By the last claim we may assume that $V(C)=\{v_1,v_2,v_3\}$. Since $\Omega$ is 4-connected, $\Omega-V(C)$ is connected and 
\begin{itemize}
\item[(1.4)] for each leaf block $B$ of $\Omega-V(C)$, we have $V(C)\subseteq N_{\Omega}(B-w)$, where $w$ is the attachment of $B$ in  $\Omega-V(C)$.
\end{itemize} 
Since each vertex in $V(B)-\{w\}$ is adjacent to at least two vertices in $C$, we have 
\begin{itemize}
\item[(1.5)] when $B$ is isomorphic to $K_3$, there is a vertex $v_i$ in $C$ such that $\Omega[(V(B)-\{w\})\cup \{v_i\}]$ is an unbalanced triangle. 
\end{itemize} 

Assume that $\Omega-V(C)$ has at least three leaf blocks $B_1,B_2,B_3$. 
Let $w$ be the attachment of $B_1$ in  $\Omega-V(C)$. 
By (1.2), (1.4) and (1.5), there is an unbalanced triangle $C'$ in $C \cup (B_1-w)$ avoiding a vertex $v_i$ of $C$. 
Since $B_1$ is a leaf block, there is a path $P$ in $\Omega-V(C)$ between $B_2$ and $B_3$ avoiding $B_1-w$. 
By (1.4) there is a cycle using $v_i$ (and no other vertex from $C$) together with $P$ and vertices from $B_2$ and $B_3$. Such cycle has length at least four, so it is unbalanced (by (1.1)) and it is disjoint from $C'$, a contradiction.
Hence, $\Omega-V(C)$ has exactly two leaf blocks $B_1$ and $B_2$. 


Assume that $B_1$ and $B_2$ are isomorphic to $K_2$. Let $p_1$ and $p_2$ be the two degree-1 vertices in $\Omega-V(C)$. Since $|V(\Omega)|\geq7$, the graph $\Omega-V(C)$ has an internal block. When there is a vertex $v_i$ in $C$ such that $\Omega[(V(\Omega)-V(C)-\{p_1,p_2\})\cup \{v_i\}]$ is unbalanced, by (1.4) the graph $\Omega$ has two disjoint unbalanced cycles (one being a 4-cycle with $p_1,p_2$ and the vertices in $V(C)-\{v_i\}$). We show that this implies that $\Omega-V(C)$ is a path with exactly three edges. Suppose this is not the case. First note that $\Omega-V(C)$ has at least three blocks, because $|V(\Omega)|\geq 7$. Then there is a vertex $q$ of degree two in $\Omega-(V(C)\cup V(B_1)\cup V(B_2))$. Since $\Omega$ is 4-connected, $q$ has two neighbours in $C$. 
Let $w_1$ and $w_2$ be the neighbours of $p_1$ and $p_2$ in $\Omega-V(C)$ respectively. 
Then $w_1$ and $w_2$ each have degree at most three in $\Omega - V(C)$, so they each have a neighbour in $C$. Let $t_1$ and $t_2$ be the neighbours of $w_1$ and $w_2$ respectively.
If $t_1=t_2$, then $\Omega[(V(\Omega)-V(C)-\{p_1,p_2\})\cup \{t_1\}]$ is unbalanced, a contradiction. 
Thus $t_1\neq t_2$ and we may assume that $t_1$ is a neighbour of $q$. Let $P$ be a $(w_1,q)$-path in $\Omega-V(C)$. Then the cycle $P\cup \{w_1t_1,t_1q\}$ is balanced (else we again have that $\Omega[(V(\Omega)-V(C)-\{p_1,p_2\})\cup \{t_1\}]$ is unbalanced), so this cycle is in fact a balanced triangle. 
Therefore, by (1.2) the triangle $t_1p_1w_1t_1$ is unbalanced. Let $P'$ be a $(q,p_2)$-path in $\Omega-V(C)$. Then the theta induced by $P'$ and $V(C)-\{t_1\}$ contains a cycle of length at least four (hence unbalanced by (1.1)) which is disjoint from the unbalanced triangle $t_1p_1w_1t_1$, a contradiction.
So $\Omega-V(C)$ is a path with exactly three edges. 
Let $\Omega-V(C)=p_1p_3p_4p_2$. Since $|\delta_{\Omega}(p_3)|, |\delta_{\Omega}(p_4)|\geq4$ and there is no $v_i$ in $C$ such that $v_ip_3p_4v_i$ is unbalanced, by (1.2) the graph $\Omega-\{p_1,p_2\}$ contains a spanning 4-wheel whose center $v_1$ is in $V(C)$ and $v_1p_3p_4v_1$ is balanced. Then $v_1p_1p_3v_1$ is unbalanced by (1.2), a contradiction as $v_1p_1p_3v_1$ is disjoint from a 4-cycle contained in $\Omega-\{v_1,p_1,p_3\}$. Hence, by symmetry we may assume that $B_1$ is isomorphic to $K_3$.

When $B_2$ is isomorphic to $K_3$ or $B_1$ and $B_2$ have no common vertex, by (1.4) and (1.5) the graph $\Omega$ has two disjoint unbalanced cycles. So $B_2$ is isomorphic to $K_2$ and shares a common vertex $u_1$ with $B_1$, implying that $\Omega-V(C)=B_1\cup B_2$. When $u_1$ has two neighbours in $C$, by (1.4) and (1.5) the graph $\Omega$ has  two disjoint unbalanced cycles. So by symmetry we may assume that $N_{\Omega}(u_1)\cap V(C)=\{v_1\}$. 
Set $\{u_2,u_3\}=V(B_1)-\{u_1\}$ and $\{u\}=V(B_2)-\{u_1\}$. When $v_2$ is adjacent with $u_2$ and $u_3$, by (1.4) we have that $v_2u_2u_3v_2$ and $v_1u_1uv_3v_1$ are disjoint unbalanced cycles. Hence, by symmetry at most one vertex of $\{u_2,u_3\}$ is adjacent with $v_2$ or $v_3$. Since $\Omega$ is 4-connected, by symmetry we may assume that $\{v_1,v_i\}=N_{\Omega}(u_i)\cap V(C)$ for each $1\leq i\leq 2$. Therefore $\Omega$ is isomorphic to the graph on the right of Figure~\ref{fig:sporadic}. 
Since each triangle disjoint from a 4-cycle is balanced, $\Omega$ is a tricoloured graph. 
\end{proof}

\begin{figure}[h!]
\begin{center}
\includegraphics[page=15,height=4.2cm]{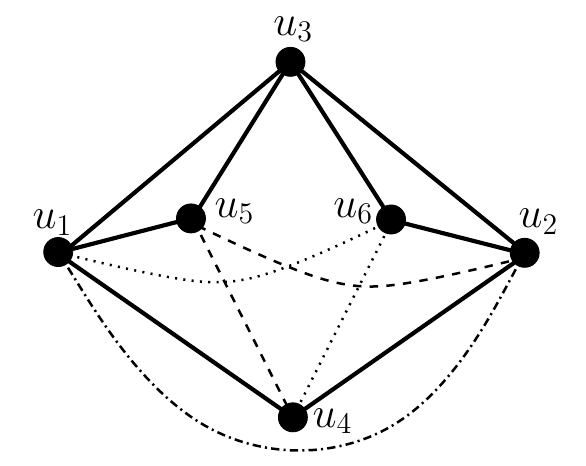}
\includegraphics[page=13,height=4.2cm]{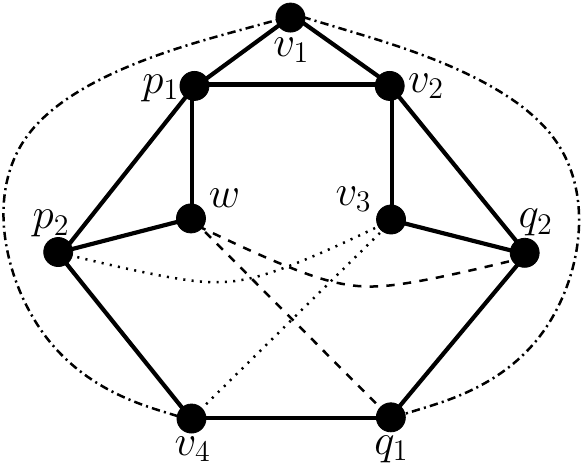}
\includegraphics[page=14,height=4.2cm]{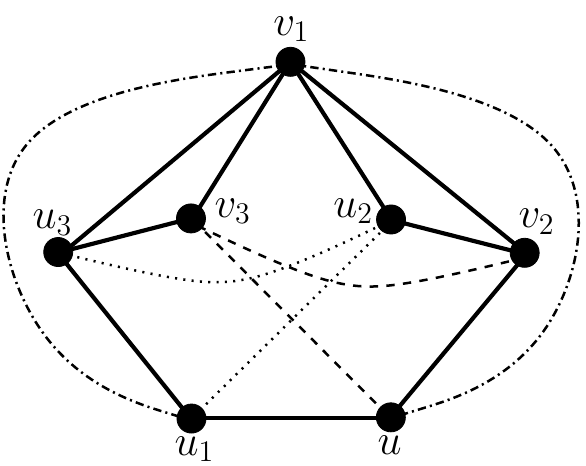}
\caption{Graphs appearing in the proof of \lref{small-balanced-sub}.}
\label{fig:sporadic}
\end{center}
\end{figure}

\begin{lemma}\label{lem:balanced-sub}
Suppose that $\Omega$ is a simple 4-connected tangled biased graph
with at least six vertices. Then either $\Omega$ is a criss-cross,
or a projective planar biased graph with a special triple, or a tricoloured graph, or there
exists a subgraph $\Omega'$ of $\Omega$ that is 2-connected,
spanning (i.e. $V(\Omega')=V(\Omega)$) and balanced.
\end{lemma}

\begin{proof}
By \tref{Lovasz}, if $\Omega$ has no balanced cycle then it either has at most 5 vertices or it is not 4-connected.
Thus $\Omega$ has at least one balanced cycle and we may choose a subgraph $\Omega'$ with the following properties:
\begin{itemize}
    \item[(O1)] $\Omega'$ is 2-connected;
    \item[(O2)] $\Omega'$ is balanced;
    \item[(O3)] subject to (O1) and (O2), $|V(\Omega')|$ is maximised;
    \item[(O4)] subject to (O1), (O2) and (O3), $|E(\Omega')|$ is maximised.
\end{itemize}

By \lref{small-balanced-sub}, we may assume that $\Omega'$ has at least four vertices.
If $\Omega'$ is spanning, then we are done. Thus we may assume that this is not the case.
Properties (O3) and (O4) imply immediately that
\begin{itemize}
    \item[$\diamondsuit$] if $P$ is an $\Omega'$-path (with endpoints $u,v$), then every cycle formed by $P$ and a $(u,v)$-path  in $\Omega'$ is unbalanced.
\end{itemize}

When $\Omega-V(\Omega')$ has two components $G_1$ and $G_2$, let $V_i$ be the set of vertices in $\Omega'$ adjacent with $G_i$, for each $1\leq i\leq 2$. Since $\Omega$ is 4-connected, $|V_1|,|V_2|\geq4$. By Lemma \ref{3-ele} and Property $\diamondsuit$ there are two disjoint unbalanced cycles in $\Omega$, a contradiction. So $\Omega-V(\Omega')$ is connected. 

Next suppose that $\Omega-V(\Omega')$ has at least two blocks.
Let $B_1$ be a leaf block, let $w_1$ the attachments of $B_1$ in $\Omega-V(\Omega')$ and let $H:=\Omega-V(\Omega')-(V(B_1)- \{w_1\})$. Let $V_1$ be the subset of vertices in $\Omega'$ adjacent with $B_1-w_1$ and let $V_2$ be the subset of vertices in $\Omega'$ adjacent with $H$. 
Since $\Omega$ is 4-connected, $|V_1|,|V_2|\geq3$ and $|V_1\cup V_2|\geq4$ (since $\Omega'$ has at least four vertices). Then by Lemma \ref{3-ele} and Property $\diamondsuit$ there are two disjoint unbalanced cycles in $\Omega$, a contradiction. Therefore $\Omega-V(\Omega')$ is 2-connected.

First we consider the case that $V(\Omega)-V(\Omega')$ has at least two vertices. When $V(\Omega)-V(\Omega')$ has exactly two vertices $w_1,w_2$, let $V_i$ be the neighbours of $w_i$ in $\Omega'$
 for each $1\leq i\leq 2$. Since $|V_1|,|V_2|\geq3$ and $|V_1\cup V_2|\geq4$, by Lemma \ref{3-ele} and Property $\diamondsuit$ there are disjoint unbalanced cycles in $\Omega$, a contradiction. Now assume that $V(\Omega)-V(\Omega')$ has exactly three vertices $w_1,w_2,w_3$. For $i=1,2,3$, let $V_i$ be the neighbours of $w_i$ in $\Omega'$. 
 Since $V(\Omega)-V(\Omega')$ is 2-connected, the graph $\Omega-V(\Omega')$ is a triangle (with possibly some edges doubled) and there are distinct vertices $u_1,u_2,u_3,u_4$ such that $u_1,u_3 \in V_1$ and $u_2,u_4\in V_2$. Since $\Omega$ has no disjoint unbalanced cycles, $\Omega'$ has no $(u_1-u_3,u_2-u_4)$-linkage. Theorem 4.1 and Lemma 4.5 imply that $(\Omega', (u_1,u_2,u_3,u_4))$ is 3-planar and $u_1,u_2,u_3,u_4$ appear in a cycle $C$ of $\Omega'$ in this order. Let $u$ be a vertex in $\Omega'$ adjacent with $w_3$. By symmetry we may further assume that $\Omega'$ has a path joining $u$ and $C[u_1,u_2)$ and disjoint from $C[u_2,u_3,u_4,u_1)$. Then, since $w_3$ is adjacent to $w_1$ in $\Omega-\Omega'$, $\Omega$ has two disjoint unbalanced cycles by Property $\diamondsuit$, a contradiction. So $\Omega-V(\Omega')$ has at least four vertices. 

Let $f_1,f_2,f_3,f_4$ be disjoint edges joining $\Omega'$ and $\Omega-V(\Omega')$. Set $f_i=u_iv_i$ with $u_i$ in $\Omega'$ and $v_i$ in $\Omega-V(\Omega')$ for each $1\leq i\leq 4$. If for each partition $(\{u_i,u_j\}, \{u_s,u_t\})$ of $\{u_1,u_2,u_3,u_4\}$ there is a $(u_i-u_j, u_s-u_t)$-linkage in $\Omega'$, then $\Omega$ obviously has two disjoint unbalanced cycles, a contradiction. So by symmetry we may assume that $\Omega'$ has no $(u_1-u_3, u_2-u_4)$-linkage. Theorem 4.1 and Lemma 4.5 imply that $(\Omega', (u_1,u_2,u_3,u_4))$ is 3-planar and $u_1,u_2,u_3,u_4$ appear in a cycle $C$ of $\Omega'$ in this order. Since $\Omega$ has no disjoint unbalanced cycles, $\Omega-V(\Omega')$ has no $(v_1-v_2,v_3-v_4)$-linkage. Using Theorem 4.1 and Lemma 4.5 again, $(\Omega-V(\Omega'), (v_1,v_3,v_2,v_4))$ is 3-planar and $v_1,v_3,v_2,v_4$ appear in a cycle $C'$ of $\Omega-V(\Omega')$ in this order. Then $C[u_4,u_1]\cup\{f_1,f_4\}\cup C'[v_4,v_1]$ and $C[u_2,u_3]\cup\{f_2,f_3\}\cup C'[v_2,v_3]$ are disjoint unbalanced cycles, a contradiction. 

Finally consider the case that $V(\Omega)-V(\Omega')$ has only one vertex $w$. We will show in this case that either $\Omega$ is a criss-cross graph or it is projective planar with a special triple.
Property $\diamondsuit$ implies immediately that
\begin{itemize}
    \item[(1.1)] every cycle formed by two edges incident with $w$ and a path in $\Omega'$ is unbalanced.
\end{itemize}

Since $\Omega$ is 4-connected, $w$ has at least four distinct neighbours in $\Omega'$. Let $u_1,u_2,u_3,u_4$ be such neighbours and let $e_i$ be a $(w,u_i)$-edge, for every $i \in [4]$.
Let $U$ be the set of edges in $E(\Omega)-E(\Omega')$ which are not incident with $w$.
Assumption (O4) implies that, for every $f \in U$, every $f$-cycle for $\Omega'$ is unbalanced.
Pick any $f=xy \in U$. Then $\Omega'$ does not contain a $(u_i-u_j,x-y)$-linkage, for any choice of distinct $i,j \in [4]$ with $u_i, u_j \notin \{x,y\}$. Therefore, \lref{link3+} implies that one of the following occurs:
\begin{itemize}
    \item[(a1)] $x,y \in \{u_1,u_2,u_3,u_4\}$, or
    \item[(a2)] $\{x,y\}$ is a 2-vertex-cut of $\Omega'$ separating $u_i$ from $u_j$, for all distinct $i,j \in [4]$ with $u_i, u_j \notin \{x,y\}$.
\end{itemize}

Next we show that case (a1) holds for every choice of $f \in U$. Suppose that case (a2) occurs. If we are not also in case (a1), then we may assume that $u_1,u_2,u_3 \notin \{x,y\}$.
Let $B_1,\ldots,B_k$ be $\{x,y\}$-bridges containing $\{u_1,u_2,u_3,u_4\}-\{x,y\}$, labelled so that $u_i \in B_i$ for $i\in[k]$. Since $\{w,x,y\}$ is not a 3-vertex-cut of $\Omega$, there exists an edge $f' \in U$ with endpoint $x',y'$, such that $x' \in B_1-\{x,y\}$ and $y' \in B-\{x,y\}$, for some $\{x,y\}$-bridge $B \neq B_1$. 
Then $\Omega'$ contains an $(x'-y',u_3-u_4)$-linkage, a contradiction. It follows that case (a1) occurs for all $f \in U$, i.e.
\begin{itemize}
    \item[(1.2)] every edge in $U$ has both endpoints in $\{u_1,u_2,u_3,u_4\}$.
\end{itemize}

Since $w$ is not a blocking vertex, there exists an edge $f_1 \in U$ with endpoints, say, $u_1$ and $u_3$.
Then $\Omega'$ does not contain a $(u_2-u_4,u_1-u_3)$-linkage, so $(\Omega',(u_1,u_2,u_3,u_4))$ is 3-planar.
By \lref{link4}, there exists a cycle $C$ of $\Omega'$ which contains $u_1,u_2,u_3,u_4$ in this circular order.

Next we show that
\begin{itemize}
    \item[(1.3)] the only neighbours of $w$ are $u_1,u_2,u_3,u_4$.
\end{itemize}
Suppose to the contrary that $w$ has a fifth neighbour $u_5$. If $u_5$ attaches to $C$ at a vertex $z \neq u_1,u_3$, then there is an unbalanced cycle (using $w,u_5$ and one of $u_2$ or $u_4$) which is disjoint from one of the two unbalanced cycles in $C \cup \{f_1\}$. It follows that $\{u_1,u_3\}$ is a 2-vertex-cut of $\Omega'$ separating $u_5$ from $u_2$ and $u_4$. Then (1.2) implies that $\{w,u_1,u_3\}$ is a 3-vertex-cut of $\Omega$, a contradiction.

Fact (1.3) implies in particular that
\begin{itemize}
    \item[(1.4)] $(\Omega',(u_1,u_2,u_3,u_4))$ is planar.
\end{itemize}
Indeed, if this is not the case, then, for some $k \in [3]$, $\Omega'$ contains a $k$-separation with none of $u_1,u_2,u_3,u_4$ in the interior. Then this is also a $k$-separation in $\Omega$, a contradiction.
Next we show that
\begin{itemize}
    \item[(1.5)] there is no edge in $U$ parallel to $f_1$.
\end{itemize}
Suppose to the contrary that there is an edge $f_2 \in U$ parallel to $f_1$. Since $\Omega$ is simple, the cycle $\{f_1,f_2\}$ is unbalanced.
Therefore $\{u_1,u_3\}$ intersects every $(u_2-u_4)$-path in $\Omega'$, i.e. $\{u_1,u_3\}$ is a 2-vertex-cut of $\Omega'$ separating $u_2$ from $u_4$. Properties (1.2) and (1.3), and the fact that $\Omega$ is 4-connected, imply that there are exactly two $\{u_1,u_3\}$-bridges in $\Omega'$, one containing $u_2$ and the other containing $u_4$. Since $\Omega$ has more than five vertices, one of these bridges (say the one containing $u_2$) has at least one vertex other than $u_1,u_2$ and $u_3$. Therefore $\{u_1,u_2,u_3\}$ is a 3-vertex-cut of $\Omega$, a contradiction.
Next we show that
\begin{itemize}
    \item[(1.6)] if $f_1$ is the only edge in $U$ then $\Omega$ is projective planar with a special triple.
\end{itemize}
Suppose that $U=\{f_1\}$. First we note that there are no edges in $\Omega$ that are parallel to either $e_2$ or $e_4$. In fact, if  there is an edge $e_2'$ parallel to $e_2$, then $\{e_2,e_2'\}$ is an unbalanced cycle disjoint from the unbalanced cycle $C[u_3,u_4,u_1]\cup \{f_1\}$.
Then $\Omega$ is projective planar with a special triple, where (following the terminology in the definition of projective planar biased graph with a special triple given in \sref{ppst}) we have that: $H=\Omega' \cup \{e_2\}$, $x=w$, $y_1=u_1$, $y_2=u_3$, $X=\{u_4\}$ and $f=f_1$.

We conclude by showing that
\begin{itemize}
    \item[(1.7)] if $U$ contains at least two edges, then $\Omega$ is a criss-cross.
\end{itemize}

Suppose that $U$ contains an edge $f_2$ other than $f_1$. By (1.2)
and (1.5), we may assume by symmetry that the endpoints of $f_2$ are
either $u_2$ and $u_4$ or $u_1$ and $u_2$. In the latter case
however, $C \cup \{f_2,e_3,e_4\}$ contains two disjoint unbalanced
cycles. Thus $f_2=u_2u_4$. 
In the proof of (1.6) we showed that the assumption that $U$ contains an edge $f_1=u_1u_3$ implies that there are no edges parallel to $e_2$ or $e_4$ in $\Omega$. The same argument applied to the case where $U$ contains edge $f_2=u_2u_4$ shows that there are also no
edges parallel to $e_1$ or $e_3$. 
It remains to show that the triangles
$\{e_1,e_3,f_1\}$ and $\{e_2,e_4,f_2\}$ are balanced. Suppose that
the triangle $\{e_1,e_3,f_1\}$ is unbalanced. Because of edge $f_2$,
this implies that $\{u_1,u_3\}$ intersect every $(u_2,u_4)$-path in
$\Omega'$. Thus $\{u_1,u_3\}$ is a 2-vertex-cut of $\Omega'$. Since
$|V(\Omega')|=|V(\Omega)|-1\geq 5$, this implies that at least one
of $\{u_1,u_2,u_3\}$ and $\{u_1,u_3,u_4\}$ is a 3-vertex-cut of
$\Omega$, a contradiction. We conclude that $\Omega'$ is indeed a
criss-cross. 
\end{proof}

\begin{lemma}\label{lem:planar}
Suppose that $\Omega$ is a simple 4-connected tangled signed graph
with at least six vertices. Suppose moreover that $\Omega$ is not a
criss-cross, a tricoloured graph, a projective planar biased graph with a special triple,
or a fat triangle. Then there is a maximal $2$-connected spanning
balanced subgraph $\Omega'$ of $\Omega$ such that
$(\Omega',(x_1,\ldots,x_m,y_1,\ldots,y_m))$ is planar (where
consecutive vertices may be repeated), where
$E(\Omega)-E(\Omega')=\{x_iy_i\ |\ i \in [m]\}$.
\end{lemma}
\begin{proof}
Let $\Omega=(G,\bal)$ be a simple 4-connected tangled biased graph.
By \lref{balanced-sub}, we have that

\begin{enumerate}
    \item[(A1)] $\Omega$ contains a 2-connected spanning balanced subgraph $\Omega'$.
\end{enumerate}

For the remainder of the proof  we choose $\Omega'$ as in (A1) to be edge-maximal and we set $U:=E(\Omega)-E(\Omega')$.
Thus

\begin{enumerate}
    \item[(A2)] For every $e \in U$, every $e$-cycle for $\Omega'$ is unbalanced.
\end{enumerate}

\noindent
First we show that

\begin{enumerate}
    \item[(A3)] $U$ contains at least two independent edges.
\end{enumerate}

If (A3) does not hold, then either $G[U]$ is a star or there exist vertices $x_1,x_2,x_3$ in $G$ such that every edge in $U$ has both endpoints in $\{x_1,x_2,x_3\}$. In the first case $\Omega$ has a blocking vertex, in the second it is a fat triangle.

For the remainder of the proof we let $U'=\{f_i=x_iy_i\ |\ i\in[n]\}$ be a maximum-sized subset of $U$ of pairwise independent edges. The next result follows immediately from \lref{link5}.

\begin{claim}\label{cl:ordered}
Up to a reordering of $[n]$, $(\Omega',(x_1,\ldots,x_n,y_1,\ldots,y_n))$ is 3-planar.
\end{claim}
For every $i \in [n]$, define sets:
\[\begin{aligned}
Y_i&=\{y \in V(\Omega)\ |\ x_iy \in U\},\textrm{and}\\
X_i&=\{x \in V(\Omega)\ |\ xy_i \in U\}.
\end{aligned}\]
Note that $x_i \in X_i$ and $y_i \in Y_i$ for every $i \in [n]$.

\begin{claim}\label{cl:planar}
Up to a reordering of $[n]$, $(\Omega',(X_1,\ldots,X_n,Y_1,\ldots,Y_n))$ is 3-planar.
\end{claim}

\begin{cproof}
By \clref{ordered}, $(\Omega',\zA,(x_1,\ldots,x_n,y_1,\ldots,y_n))$ is 3-planar for some $\zA$.
Let $H=\proj(\Omega',\zA)$.
For any vertex $w$, let $w^*=w$ if $w \in V(H)$ and $w^*$ be an arbitrary vertex in $N_G(A)$ if $w \in A$ for some $A \in \zA$.
Let $C$ be the face boundary of $H$ containing $x_1,\ldots,x_n,y_1,\ldots,y_n$.
For every $i \in [n]$, let $X_i'=X_i \cap V(C)$ and $Y_i'=Y_i \cap V(C)$.
Among all possible choices for $\zA$, pick one such that
$|X_1' \cup \cdots \cup X_n' \cup Y_1' \cup \cdots \cup Y_n'|$ is maximized.

Since $U'$ is maximal, every edge in $U$ has at least one end in $\{x_1,\ldots,x_n,y_1,\ldots,y_n\}$.
We easily see from  \lref{link2} that
\begin{itemize}
    \item[(A4)] For every $y \in Y_i$, $y^*$ only attaches to vertices in $V(C[y_{i-1},y_i,y_{i+1}])$ (where we take $y_0=x_n$ and $y_{n+1}=x_1$). A symmetric statement holds for vertices in  $X_i$.
\end{itemize}
Therefore, the claim is easily seen to hold when $X_i, Y_i \subseteq V(C)$ for every $i \in [n]$.
Suppose that $x_1y \in U$ for some $y \notin V(C)$.

First assume that $y^* \notin V(C)$. Since $y^*$ only attaches to vertices in $C[x_n,y_1,y_2]$ in $C$, by planarity there exists a 2-vertex-cut $\{a,b\}\subseteq V(C[x_n,y_1,y_2])$ of $H$ such that $y^*$ and $C[y_3,\ldots,y_n,x_1,\ldots,x_{n-1}]$ are in different $\{a,b\}$-bridges. Let $B$ be the $\{a,b\}$-bridge containing $y^*$. Since $\{x_1,a,b\}$ is not a 3-vertex-cut of $\Omega$, there exists some vertex $z^* \in B-\{a,b\}$ such that $z \in X_2\cup \cdots \cup X_n\cup Y_1 \cup \cdots \cup Y_n$. Then (A4) implies that either $z \in X_n$ (and $a,b \in V(C[x_n,y_1])$) or $z \in Y_2$ (and $a,b \in V(C[y_1,y_2])$). By symmetry we may assume the latter. If, for some $i$, a vertex $w \in Y_i$ is in $C(a,b)$, then $\Omega'$ contains an $(x_i-w,y-x_1)$-linkage (or an $(x_i-w,z-x_2)$-linkage when $i=1$), a contradiction. By a similar argument we conclude that $C(a,b)$ does not contain any vertex in $X_1 \cup \cdots \cup X_n \cup Y_1 \cup \cdots \cup Y_n$. Let $B'$ be the $\{a,b\}$-bridge in $\Omega'$ corresponding to $B$. Define a new set $\zA'$ from $\zA$ by
$$\zA'=(\zA-\{A\in\zA\ |\ N_{\Omega'}(A)\subseteq B'\})\cup\{B'-\{a,b,y\}\}.$$
Then $(\Omega',\zA',(x_1,\ldots,x_n,y_1,y,y_2,\ldots,y_n))$ is 3-planar and $\zA'$ contradicts the choice of $\zA$.

Now suppose that $y^* \in V(C)$. Since $y \notin V(C)$, this implies that $y \in A$ for some $A \in \zA$.
If some vertex in $N_G(A)$ is not in $V(C)$, then we may replace $y^*$ with this vertex, and reduce to the previous case.
Now suppose that $N_G(A) \subset V(C)$. Since $y^*$ is chosen arbitrarily in $N_G(A)$, property (A4) implies that $N_G(A) \subset V(C[x_n,y_1,y_2])$. Let $B$ be the $N_G(A)$-bridge in $\Omega'$ containing $y$.
Since $\Omega'$ is 2-connected, $N_G(A)$ contains at least two vertices. Let $N_G(A)=\{a,b\}$ if $N_G(A)$ has size two and $N_G(A)=\{a,b,c\}$ otherwise, chosen so that $c$ is in the $(a,b)$-path of $C$ avoiding $x_1$. Since $a$ and $b$ are adjacent in $H$, the planarity of $H$ implies that $\{a,b\}$ is a 2-vertex-cut of $\Omega'$. Since $\{a,b,x_1\}$ is not a 3-vertex-cut of $\Omega$, there exists some vertex $z \in B-\{a,b\}$ such that $z \in X_2\cup \cdots \cup X_n\cup Y_1 \cup \cdots \cup Y_n$. Then (A4) implies that either $z \in X_n$ (and $a,b \in V(C[x_n,y_1])$) or $z \in Y_2$ (and $a,b \in V(C[y_1,y_2])$). By symmetry we may assume the latter. From here we may proceed in a similar manner to the previous case.
\end{cproof}

Since $\Omega$ is 4-connected and each edge in $U$ is incident with $X_1\cup \cdots \cup X_n\cup Y_1 \cup \cdots \cup Y_n$, \clref{planar} implies that $\Omega'$ is planar and all the vertices incident with edges in $U$ are on a same face boundary $C$. Since $\Omega'$ is 2-connected, $C$ is a cycle. Therefore, if $U=\{f_1,\ldots,f_m\}$, where $f_i=x_iy_i$ for every $i \in [m]$, then (by relabelling) we may assume that $(\Omega',(x_1,\ldots,x_m,y_1,\ldots,y_m))$ is planar (where some consecutive vertices may be repeated).
\end{proof}

\section{Proof of \tref{main-thm}}\label{sec:main_proof}

First we show that \tref{main-thm} holds when $\Omega$ has at most five vertices.
\begin{lemma}\label{lem:small}
Let $\Omega$ be a simple tangled biased graph with at most five vertices. Then \tref{main-thm} holds.
\end{lemma}
\begin{proof}
Since $\Omega$ has no blocking vertex, it has at least three vertices.
Moreover, by Lemmas~\ref{lem:1sep} and~\ref{lem:2sep} we may assume that $\Omega$ is 2-connected.
If $\Omega$ has exactly three vertices, then it is a fat triangle.
Assume that $\Omega$ has exactly four vertices. If some vertex $v$ of $\Omega$ is not adjacent any parallel edges, define $H$ to be a maximal balanced subgraph of $\Omega$ containing all the edges incident with $v$. Let $v_1,v_2$ and $v_3$ be the vertices of $\Omega$ other than $v$. Then there is at least one $v_1v_2$-edge not in $H$ (and forming an unbalanced cycle with $H$), otherwise $v_3$ would be a blocking vertex of $\Omega$. Similarly for the other two pairs of vertices $v_i$ and $v_j$ with $i\neq j$. It follows that $\Omega$ is a fat triangle. So we may assume that every vertex of $\Omega$ is adjacent with some parallel edges. Moreover, since $\Omega$ is simple and has no two disjoint unbalanced cycles, there is a vertex $v$ of $\Omega$ such that all parallel edges are incident with $v$ and for any other vertex $w$ of $\Omega$ there is a pair of parallel edges between $v$ and $w$. Then $\Omega$ is a generalized wheel (with center $v$).

Finally suppose that $\Omega$ has exactly five vertices $v_1,v_2,v_3,v_4,v_5$.
\lref{4conn} implies that either $\Omega$ satisfies cases (T1) or (T3) or all the vertices of $\Omega$ are pairwise adjacent, i.e. $||\Omega||$ is obtained from $K_5$ by possibly adding parallel edges.

Assume that (T2) does not hold. By symmetry we may assume that there are two unbalanced 2-cycles in $\Omega$ between vertex $v_1$ and vertices $v_2$ and $v_3$. Then triangles $v_2v_4v_5v_2$ and $v_3v_4v_5v_3$ are balanced. Thus, the 4-cycle $v_5v_2v_4v_3v_5$ is balanced.
If there is also an unbalanced 2-cycle between $v_2$ and $v_3$, then $\Omega$ is a fat triangle.
So now suppose that there is no such unbalanced 2-cycle.
Since the 4-cycle $v_5v_2v_4v_3v_5$  is balanced, and it forms a theta subgraph with the edge $v_2v_3$, triangles $v_2v_3v_4v_2$ and $v_2v_3v_5v_2$ are either both balanced or both unbalanced. Since $v_1$ is not a blocking vertex, they are both unbalanced.
It follows that there are no other unbalanced 2-cycles in $\Omega$. Now if the triangle $v_1v_4v_5v_1$ is balanced, then $\Omega$ is a fat triangle. Otherwise $\Omega$ is a generalized wheel (with center $v_1$ and, in the definition of generalized wheel, vertices $z_1=v_2$ and $z_2=v_3$.)
\end{proof}

Next we prove a useful lemma to identify when $\Omega$ is signed-graphic.

\begin{lemma}\label{lem:2balanced}
Let $\Omega'$ be a maximal balanced subgraph of a biased graph $\Omega$. Suppose that $\Omega$ contains no two disjoint unbalanced cycles.
If $F$ is 2-balanced with respect to $\Omega'$ then $\Omega' \cup F$ is a signed graph with signature $F$.
\end{lemma}

\begin{proof}
It suffices to prove the statement for each connected component of $\Omega$. Thus we assume that $\Omega$ is connected.
Let $C$ be any cycle in $\Omega' \cup F$.
To prove the result it suffices to show that $C$ is balanced if and only if $|F \cap E(C)|$ is even.
We proceed by induction on $k=|F \cap E(C)|$.
If $k=0$, then $C$ is a cycle of $\Omega'$, which is balanced, so $C$ itself is balanced.
If $k=1$ then $C$ is unbalanced by the maximality of $\Omega'$.
If $k=2$, then $C$ is balanced since $F$ is $2$-balanced.
Now suppose that $k \geq 3$.
Let $P_1,\ldots,P_k$ be the components of $C \bs F$. Each $P_i$ is a path of $\Omega'$ (possibly comprising a single vertex) and by relabelling we may assume that $P_1, \ldots, P_k$ appear in this order along $C$.
For the remainder of the proof, for distinct $i,j \in [k]$, we say that $P_i$ connects to $P_j$ if there exists an $(x_i-x_j)$-path $Q$ in $\Omega'$ with $x_i \in V(P_i)$ and $x_j \in V(P_j)$ and $Q$ is internally disjoint from all of $P_1, \ldots, P_k$. In this case the path $Q$ {\em connects} $P_i$ to $P_j$.

First suppose that $k$ is odd. Since $\Omega'$ is maximal and $\Omega$ is connected, $\Omega'$ is also connected. Thus there exists a path $Q$ connecting $P_i$ and $P_j$ for some distinct $i,j \in [k]$. Then $C \cup Q$ is a theta graph. Let $C_1$ and $C_2$ be the two cycles in $C \cup Q$ containing $Q$. Note that $C_1 \cap F$ and $C_2 \cap F$ are both nonempty. Since $Q$ does not contain any edge in $F$, we have that $k_1=|C_1 \cap F|<k$, $k_2=|C_2 \cap F|<k$ and one of $k_1$ and $k_2$ is odd, and the other is even. By induction, one of $C_1$ and $C_2$ is unbalanced and the other is balanced. Therefore, by the theta property $C$ is unbalanced.

Now suppose that $k$ is even. Suppose that some $P_i$ connects to some $P_{i+2\ell}$ for some number $\ell$, through a path $Q$. Then we may apply a similar argument to the one above, where now $k_1$ and $k_2$ are even. Thus the two cycles in $C \cup Q$ using $Q$ are balanced and the theta property implies that $C$ is balanced as well. To complete the proof it remains to show that we may always find such $P_i$ and $P_{i+2\ell}$.

If every $P_i$ connects to only one other $P_j$, then $\Omega'$ is disconnected (since $k \geq 4$).
So without loss of generality we may assume that $P_1$ connects to $P_{j_1}$ and to $P_{j_2}$, where $j_1 < j_2$.
If one of $j_1$ or $j_2$ is odd then we are done. So assume that $j_1$ and $j_2$ are even.
Choose an odd $j_3$ with $j_1 < j_3 < j_2$.
Now $P_{j_3}$ connects to some other $P_{j_4}$. If $j_4$ is odd then we are done.
Let $Q_i$ be a path connecting $P_1$ to $P_{j_i}$, for $i=\{1,2\}$.
Let $Q_3$ be a path connecting $P_{j_3}$ to $P_{j_4}$.
For $i \in [3]$, both cycles in $C \cup Q_i$ using $Q_i$ are unbalanced,
hence $C \cup Q_1 \cup Q_2 \cup Q_3$ contains two disjoint unbalanced cycles.
\end{proof}

Suppose that $\Omega$ is a tangled biased graph with a maximal $2$-connected spanning balanced subgraph $\Omega'$ as in \lref{planar}, i.e. $E(\Omega)-E(\Omega')=\{x_iy_i\ |\ i \in [m]\}$ and
$(\Omega',(x_1,\ldots,x_m,y_i,\ldots,y_m))$ is planar (where consecutive vertices may be repeated).
We say that such an $\Omega$ is a {\em projective planar} biased graph {\em based on $\Omega'$}.
Every time we refer to such a graph we will indicate by $\{f_i=x_iy_i\ |\ i \in [m]\}$ the set of edges in $E(\Omega)-E(\Omega')$.

\begin{lemma}\label{lem:hasbp}
Let $\Omega$ be a simple 4-connected projective planar tangled biased graph.
Then either
\begin{enumerate}
    \item[(1)] $\Omega$ has a blocking pair, or
    \item[(2)] $\Omega$ is a projective planar signed graph, or
    \item[(3)] $\Omega$ is projective planar with a special vertex, or
    \item[(4)] $\Omega$ is a tricoloured graph, or 
    \item[(5)] $\Omega$ is a criss-cross.
\end{enumerate}
\end{lemma}
\begin{proof}
Suppose that $\Omega$ is projective planar based on $\Omega'$. Thus $(\Omega',(x_1,\ldots,x_m,y_i,\ldots,y_m))$ is planar and $U:=\{f_i=x_iy_i\ |\ i \in [m]\}=E(\Omega)-E(\Omega')$.
Unless otherwise specified, throughout the proof we will refer to $A$-cycles (for some edge or some set $A$ of edges), meaning an $A$-cycle with respect to $\Omega'$.
By the maximality of $\Omega'$, every $f_i$-cycle is unbalanced.
Denote by $C$ the face boundary of $\Omega'$ containing $x_1,\ldots,x_m,y_1,\ldots,y_m$.

We assume that we are not in case (1) (i.e. $\Omega$ does not have a blocking pair).

\begin{claim}\label{cl:indep}
For every $f_i \in U$ there exists an edge $f_j \in U$ that is independent from $f_i$ and $f_{i+1}$.
\end{claim}
\begin{cproof}
For simplicity, assume that $i=1$. Since $x_1,x_2$ is not a blocking pair, there is an edge $f_k$ not incident with $x_1,x_2$. If $f_k$ is not incident with $y_1$ and $y_2$ then we are done. Otherwise we may assume that $y_k=y_2$.
Similarly, $y_1$ and $y_2$ do not form a blocking pair, so there is an edge $f_\ell$ with $y_\ell=x_1$ ($f_\ell$ has to be incident with $x_1$ because of the position of $x_k$). Now one can check that $x_1,y_2$ form a blocking pair, a contradiction.
\end{cproof}


When $\Omega$ contains two parallel edges $f$ and $f'$, since $\Omega$ is simple, $\{f,f'\}$ is an unbalanced cycle of $\Omega$ and the endpoints of $f$ intersect all unbalanced cycles of $\Omega$, thus forming a blocking pair. Therefore there are no parallel edges in $U$.

Because of  \lref{2balanced}, if $U$ is 2-balanced with respect to $\Omega'$ then $\Omega$ is a projective planar signed graph (i.e. (2) holds).
Thus for the remainder of the proof we assume that $U$ is not 2-balanced with respect to $\Omega'$.
Hence there exist $i,j \in [m]$, with $i<j$, such that some $\{f_i,f_j\}$-cycle  is unbalanced.
We may assume that, aside from $f_i$ and $f_j$, such cycle comprises an $(x_i,x_j)$-path $P_x$ and a $(y_i,y_j)$-path $P_y$ in $\Omega'$.
Since $\Omega'$ is planar, $C[x_i,x_{i+1},\ldots, x_j]$ may be obtained from $P_x$ by rerouting along balanced cycles (disjoint from $P_y$) and $C[y_i,y_{i+1},\ldots,y_j]$ may be obtained from $P_y$ by rerouting along balanced cycles (disjoint from $P_x$). Therefore, $C[x_i,x_{i+1},\ldots,x_j] \cup C[y_i,y_{i+1},\ldots,y_j] \cup \{f_i,f_j\}$ is an unbalanced cycle.

Next we show that we may choose $i$ and $j$ so that $j=i+1$. Choose $i$ and $j$ so that $j-i$ is minimised. If $j \neq i+1$ (so $f_{i+1}\neq f_j$), then $C[x_i,x_{i+1},\ldots,x_j] \cup C[y_i,y_{i+1},\ldots,y_j] \cup \{f_i,f_{i+1},f_j\}$ is a theta subgraph, thus either the cycle using $f_i$ and $f_{i+1}$ or the cycle using $f_{i+1}$ and $f_j$ is unbalanced, contradicting the choice of $i$ and $j$. It follows that $U$ contains edges $f_i$ and $f_{i+1}$ such that $C[x_i,x_{i+1}] \cup C[y_i,y_{i+1}] \cup \{f_i,f_{i+1}\}$ is unbalanced. We call such a pair $\{f_i,f_{i+1}\}$ a {\em close unbalanced pair}. In this definition we choose indices modulo $m$, so it may be that $f_1$ and $f_m$ form a close unbalanced pair. If $\{f_i,f_{i+1}\}$ is a close unbalanced pair, we denote by $C_{i,i+1}$ the unbalanced cycle $C[x_i,x_{i+1}] \cup C[y_i,y_{i+1}] \cup \{f_i,f_{i+1}\}$. 

Next we introduce a few more definitions that will come in handy for the remainder of the proof.
If a vertex $c$ is in a 2-vertex-cut of $\Omega'$, then we say that $c$ is an {\em intermediate} vertex of $\Omega'$. If a 2-vertex-cut $\{c,d\}$ of $\Omega'$ intersects all $(x_i,y_i)$-paths for every $i \in [m]$, then we say that $\{c,d\}$ is a {\em diagonal} 2-vertex-cut of $\Omega'$.
Let $\{f_i,f_{i+1}\}$ be a close unbalanced pair and let
$f_k$ be an edge in $U$ independent from $f_i$ and $f_{i+1}$ (as in \clref{indep}). Since $C_{i,i+1}$ intersects every $f_k$-cycle,
there exist vertices $c \in C[x_i,x_{i+1}]$ and $d \in C[y_i,y_{i+1}]$ such that $\{c,d\}$ is a 2-vertex-cut separating $x_k$ from $y_k$. Then $\{c,d\}$ is a diagonal 2-vertex-cut of $\Omega'$. We say that $\{c,d\}$ is a diagonal 2-vertex-cut {\em associated} with the close unbalanced pair $\{f_i,f_{i+1}\}$.

\begin{claim}\label{cl:anotherpair}
There exists at least two close unbalanced pairs.
\end{claim}
\begin{cproof}
We already showed that there exists a close unbalanced pair. To simplify notation suppose that $\{f_1,f_2\}$ is such a pair.
Let $\{c,d\}$ be a diagonal 2-vertex-cut of $\Omega'$ associated with $\{f_1,f_2\}$.
Since $\{c,d\}$ is not a blocking pair, there exists an unbalanced cycle $C'$ avoiding $c$ and $d$.
Thus $|E(C')\cap U|\geq 2$ and by the theta property we may in fact choose $C'$ so that $|E(C') \cap U|=2$.
Then again the theta property (and \clref{indep}, if $E(C') \cap U = \{f_1,f_2\}$) implies that
$U$ contains another close unbalanced pair.
\end{cproof}

\begin{claim}\label{cl:commonvtx}
Any two close unbalanced pairs share at least one vertex and at most two. If they share two, then those are the endpoint of an edge common to the two pairs.
\end{claim}
\begin{cproof}
Let $\{f_i,f_{i+1}\}$ and $\{f_j,f_{j+1}\}$ be two distinct close unbalanced pairs.
Since $C_{i,i+1}$ and $C_{j,j+1}$ intersect, at least one of the edges $f_i$ and $f_{i+1}$ shares an endpoint with one of $f_j$ and $f_{j+1}$. Thus the two pairs share at least one vertex. Now it is easy to see that the second part of the claim holds since $\Omega$ is projective planar.
\end{cproof}

\begin{claim}\label{cl:share}
Let $\{f_i,f_{i+1}\}$ and $\{f_{i+1},f_{i+2}\}$ be two close unbalanced pairs.
Then the endpoints of $f_{i+1}$ form a diagonal 2-vertex-cut of $\Omega'$.
\end{claim}
\begin{cproof}
To simplify notation we assume that $i=1$, i.e.
the two close unbalanced pairs are $\{f_1,f_2\}$ and $\{f_2,f_3\}$.
Assume that $x_2$ is not an intermediate vertex of $\Omega'$.
Then there are vertices $c\in C[x_1,x_2)$ and  $c' \in C(x_2,x_3]$
such that $c$ and $c'$ are intermediate vertices of $\Omega'$.
In particular, $x_1 \neq x_2$ and $x_2 \neq x_3$, so the only edge in $U$ that is incident with $x_2$ is $f_2$.
Then $\{c,c',y_2\}$ is a 3-vertex-cut of $\Omega$, a contradiction.
So $x_2$ is an intermediate vertex. By symmetry  $y_2$ is also an intermediate vertex.
It follows that $\{x_2,y_2\}$ is a diagonal 2-vertex-cut.
\end{cproof}

\begin{claim}\label{cl:intermediate}
Let $\{f_i,f_{i+1}\}$ and $\{f_j,f_{j+1}\}$ be two close unbalanced pairs sharing exactly one vertex (say $x_{i+1}=x_j$).
Then $j=i+2$ and $y_{i+1}$ and $y_j$ are intermediate vertices of $\Omega'$.
Moreover,  either $x_{i+1}$ is an intermediate vertex, or there exist $c\in C[x_i,x_{i+1})$ and $c'\in C(x_j,x_{j+1}]$ such that $\{c,y_{i+1}\}$ and $\{c',y_j\}$ are diagonal 2-vertex-cuts of $\Omega'$.
\end{claim}
\begin{cproof}
To simplify notation we assume that $i=1$, and the pairs share the vertex $x_2=x_j$. Then $y_2 \neq y_j$ and there are vertices $c\in C[x_1,x_2]$ and  $d \in C[y_1,y_2]$ such that $\{c,d\}$ is a diagonal 2-vertex-cut and vertices $c'\in C[x_j,x_{j+1}]$ and  $d' \in C[y_j,y_{j+1}]$ such that $\{c',d' \}$ is a diagonal 2-vertex-cut. If $y_2$ or $y_j$ is not an intermediate vertex, then $\{x_2,d,d'\}$ is a 3-vertex-cut of $\Omega'$. Thus both  $y_2$ and $y_j$ are intermediate vertices; moreover $y_2y_j$ is an edge of $\Omega'$ (so $j=3$ and the close unbalanced pair $\{f_j,f_{j+1}\}$ is in fact $\{f_3,f_4\}$). If $x_2$ is an intermediate vertex then we are done. Else, $c$ and $c'$ are distinct from $x_2$ and the claim follows.
\end{cproof}

\begin{claim}\label{cl:notint}
There exist two close unbalanced pairs with a common vertex $x$ which is not an intermediate vertex of $\Omega'$.
\end{claim}
\begin{cproof}
Suppose this is not the case.
\clref{share} and \clref{intermediate} imply that for every two close unbalanced pairs there exists a diagonal 2-vertex-cut of $\Omega'$ containing all vertices common to the two pairs.
Suppose that $\{f_i,f_{i+1}\}$ and $\{f_j,f_{j+1}\}$ are close unbalanced pairs and let $\{c,d\}$ be a diagonal 2-vertex-cut containing the vertices common to the two pairs.
Since $\{c,d\}$ is not a blocking pair, there exists a third close unbalanced pair $\{f_k,f_{k+1}\}$  such that $\{c,d\}$ does not intersect $C_{k,k+1}$. Therefore, no vertex of $\Omega'$ belongs to all three close unbalanced pairs $\{f_i,f_{i+1}\}$, $\{f_j,f_{j+1}\}$ and $\{f_k,f_{k+1}\}$.
\clref{share} and \clref{intermediate} imply that $U=\{f_i,f_{i+1},f_j,f_{j+1},f_k,f_{k+1}\}$
and by symmetry we may assume that $i<j<k$.
Thus $f_{i+1}$ shares an intermediate vertex $w_1$ with $f_j$,
$f_{j+1}$ shares an intermediate vertex $w_2$ with $f_k$,
and $f_{k+1}$ shares an intermediate vertex $w_3$ with $f_i$.
Then $w_1,w_2,w_3$ are three distinct intermediate vertices of $\Omega'$. Since  $\Omega-\{w_1,w_2,w_3\}$ is connected, we have that $\Omega'$ comprises only a, say, $\{w_1,w_2\}$-bridge and two edges $w_1w_3$ and $w_2w_3$. It follows that either $\{w_1,w_2\}$ is a blocking pair of $\Omega$ or $\{f_{k+1},f_i\}$ is a close unbalanced pair. In this case either $\Omega$ has two disjoint unbalanced cycles or a 3-vertex-cut, a contradiction.
\end{cproof}

To complete the proof it remains to show the following.

\begin{claim}
Either $\Omega$ is projective planar with a special vertex or it is a tricoloured graph.
\end{claim}
\begin{cproof}
\clref{notint} and \clref{intermediate} imply that there exist close unbalanced pairs $\{f_i,f_{i+1}\}$ and $\{f_j,f_{j+1}\}$ sharing vertex $x_{i+1}=x_j$ such that $x_{i+1}$ is not an intermediate vertex of $\Omega'$.
To simplify notation assume that $i=1$. That is, the two close unbalanced pairs are $\{f_1,f_2\}$ and $\{f_3,f_4\}$ and $x_2=x_3$ is not an intermediate vertex. By \clref{intermediate}, there exist vertices $c\in C[x_1,x_2)$ and $c'\in C(x_3,x_4]$    such that $\{c,y_2\}$ and $\{c',y_3\}$ are diagonal 2-vertex-cuts of $\Omega'$. In particular, this implies that $x_1 \neq x_2$ and $x_3 \neq x_4$. This also implies that $x_2$ is a degree-2 vertex in $\Omega'$ (adjacent to $c$ and $c'$), for otherwise $\{x_2,c,c'\}$ is a 3-vertex-cut of $\Omega$. Moreover, $cc'$ is an edge of $\Omega'$ as $x_2$ is not an intermediate vertex of $\Omega'$.
Since  $\{x_2,y_2,y_3\}$ is not a 3-vertex-cut of $\Omega$, there is only one $\{y_2,y_3\}$-bridge in $\Omega'$ (the one containing $x_2$) and $y_2y_3$ is an edge of $\Omega'$. Since we will frequently refer to this edge, we denote it as $g=y_2y_3$.
For any two $f_i,f_j \in U$ (with $i<j$), denote by $C_{i,j}$ the cycle $C[x_i,x_{i+1},\ldots,x_j] \cup C[y_i,y_{i+1},\ldots,y_j] \cup \{f_i,f_j\}$. We also denote as $C_{i,1}$ the cycle $C[y_i,y_{i+1},\ldots,y_m,x_1] \cup C[x_i,x_{i+1},\ldots,x_m,y_1] \cup \{f_1,f_i\}$, for every $1<i \leq m$.
Since any $\{f_i,f_j\}$-cycle $C'$ for $\Omega'-g$ may be obtained by rerouting $C_{i,j}$ along balanced cycles in $\Omega'-g$,
\begin{itemize}
    \item[(P1)] for any two $f_i,f_j \in U$ (with $i<j$), every $\{f_i,f_j\}$-cycle for $\Omega'-g$ has the same bias as $C_{i,j}$.
\end{itemize}
Define a relation $\sim$ on $U$ as $f_i \sim f_j$ if either $i=j$ or $\{f_i,f_j\}$ is 2-balanced for $\Omega'-g$. We claim that
\begin{itemize}
    \item[(P2)] the relation $\sim$ is an equivalence relation on $U$.
\end{itemize}
Clearly $\sim$ is reflexive and symmetric. Now suppose that $f_i \sim f_j$ and $f_j \sim f_k$. To simplify notation assume that $i<j<k$. Then $C_{i,j}$ and $C_{j,k}$ are both balanced, and they form a theta subgraph, whose third cycle is $C_{i,k}$. It follows that $C_{i,k}$ is balanced and (P1) implies that $\{f_i,f_k\}$ is 2-balanced for $\Omega'-g$.

We assign to each equivalence class for $\sim$ a colour, so that distinct classes get different colours. Then (P1) implies that if $f_i$ and $f_j$ have distinct colours, every $\{f_i,f_j\}$-cycle for $\Omega'-g$ is unbalanced.

Applying \clref{indep} to $f_2$ and $f_3$ shows that there exists an edge $f_k \in U$ not incident with $y_2,y_3$ and $x_2$. Thus, since $cc'$ is an edge of $\Omega'$, the cycle $\{f_2,f_3,g\}$ is disjoint from some $f_k$-cycle. It follows that $\{f_2,f_3,g\}$ is balanced. Moreover, since $f_2$ and $f_3$ share a vertex, $\{f_2,f_3\}$ is 2-balanced for $\Omega'$. In particular, $f_2$ and $f_3$ have the same colour.

Define the set $U'=\{f_k \in U\ |\ x_k,y_k \notin \{x_1,y_1,x_4,y_4\}\}-\{f_2,f_3\}$ (i.e. $U'$ is the set of edges in $U$ that are independent from $f_1,f_2,f_3$ and $f_4$). First suppose that $U'$ is nonempty. We claim that in this case $\Omega$ is projective planar with a special vertex. First we show that all the edges in $U-\{f_2,f_3\}$ have the same colour.
Pick any $f_k \in U'$; then $C_{k,1}$ is disjoint from $C_{3,4}$, thus $C_{k,1}$ is balanced. Similarly, $C_{4,k}$ is disjoint from $C_{1,2}$, thus $C_{4,k}$ is also balanced. It follows that every edge in $U'$ has the same colour as $f_1$ and $f_4$. A similar argument shows that any other edge in $U$ incident with $x_1,x_4,y_1$ or $y_4$ has also the same colour. Thus all the edges in $U-\{f_2,f_3\}$ have the same colour. Such colour is different from the colour of $f_2$ and $f_3$, since $\{f_1,f_2\}$ is a close unbalanced pair. It follows that in this case $\Omega$ is projective planar with a special vertex.
We therefore assume for the remainder of the proof that $U'$ is empty. This implies that all the edges in $U-\{f_2,f_3\}$ are incident to a vertex in $\{x_1,x_4,y_1,y_4\}$ (because $\Omega$ is projective planar on $\Omega'$). Up to symmetry, this leaves three possibilities.

\textbf{Case 1:}  $U-\{f_2,f_3\}$ partitions into two sets $U_1$ and $U_4$, where all the edges in $U_1$ are incident to $x_1$ and all the edges in $U_4$ are incident to $x_4$. In this case at most one of $x_1$ and $x_4$ is an intermediate vertex, for otherwise $\{x_1,x_4\}$ is a blocking pair (since $\{f_2,f_3\}$ is 2-balanced). Moreover, $x_1 \neq y_4$ (and symmetrically, $x_4 \neq y_1$), for otherwise $U_4=\{f_4\}$ and $\{x_1,x_2\}$ is a blocking pair.
If $U_1$ and $U_4$ contain no close unbalanced pair, then all edges in $U_1$ have the same colour and all edges in $U_4$ have the same colour. Else we may assume (by symmetry) that there is a close unbalanced pair $\{f_i,f_{i+1}\}$ in $U_1$. Since every $f_3$-cycle intersects $C_{i,i+1}$, we have that $x_1$ is an intermediate vertex.
Since the close unbalanced pairs $\{f_3,f_4\}$ and $\{f_i,f_{i+1}\}$ share a vertex, we have $x_i=x_4$ (i.e. $f_i$ is an edge between $x_1$ and $x_4$). Thus the edges in $U_1-\{f_i\}$ have all the same colour. Now if $U_4 \cup \{f_i\}$ is not 2-balanced for $\Omega'-g$, then $x_4$ is also an intermediate vertex, a contradiction to the fact that at most one of $x_1$ and $x_4$ is an intermediate vertex. It follows that $U_1-\{f_i\}$ and $U_4 \cup \{f_i\}$ are both
equivalence classes for $\sim$. Thus we may assume that we started with $U_1$ and $U_4$ where all the edges in $U_1$ have the same colour and all the edges in $U_4$ have the same colour.

Now if $U_1 \cup U_4$ have the same colour, then $\Omega$ is again projective planar with a special vertex.
So we may assume that $U_1$ and $U_4$ have different colours.
Such colours are both different from the colour of $f_2$ and $f_3$ (since $\{f_1,f_2\}$ and $\{f_3,f_4\}$ are close unbalanced pairs). We show that in this case $\Omega$ is a tricoloured graph.
Denote as $N_1$ the neighbour of $x_1$ via edges in $U_1$ and as $N_4$ the neighbours of $x_4$ via edges in $U_4$.
Since $U_1$ and $U_4$ have different colours, for every $f_i \in U_1$ and every $f_j \in U_4$, every $\{f_i,f_j\}$-cycle for $\Omega'-g$ intersects every $f_2$-cycle and every $f_3$-cycle. Therefore there is a 2-vertex-cut $\{d,d'\}$ of $\Omega'$ with $d\in C[y_4,\ldots,y_m,x_1]$ and $d'\in C[x_4,\ldots,x_m,y_1]$, where $\{d,d'\}$ separates $x_1$ from $N_1$ and $x_4$ from $N_4$. Now $\Omega$ is a tricoloured graph where (following the notation in the definition of tricoloured graphs given in \sref{tricoloured}) we have $I=\{1,2,3\}$ and:
\begin{itemize}
    \item $H_1=\{c\}$ if $c=d$, otherwise $H_1$ is the $\{d,c\}$-bridge of $\Omega'$ not containing $x_2$ (or $H_1=\{dc\}$ if there is only one $\{d,c\}$-bridge);
    \item $H_2$ is the $\{c,c'\}$-bridge containing $x_2$;
    \item $H_3=\{c'\}$ if $c'=d'$, otherwise $H_3$ is the $\{c',d'\}$-bridge of $\Omega'$ not containing $x_2$ (or $H_3=\{c'd'\}$ if there is only one $\{c',d'\}$-bridge);
    \item $H_4=\{d'\}$ if $y_2 = d'$, otherwise $H_4$ is the $\{d',y_2\}$-bridge of $\Omega'$ not containing $x_2$ (or $H_4=\{d'y_2\}$ if there is only one $\{d',y_2\}$-bridge);
    \item $H_5=\{g\}$;
    \item $H_6=\{d\}$ if $y_3 = d$, otherwise $H_6$ is the $\{y_3,d\}$-bridge of $\Omega'$ not containing $x_2$ (or $H_6=\{y_3d\}$ if there is only one $\{y_3,d\}$-bridge).
\end{itemize}

\textbf{Case 2:}  $U-\{f_2,f_3\}$ partitions into two sets $U_1$ and $U_4$, where all the edges in $U_1$ are incident to $y_1$ and all the edges in $U_4$ are incident to $y_4$. This case is very similar to Case 1; we include the proof for completeness.

At most one of $y_1$ and $y_4$ is an intermediate vertex, for otherwise $\{y_1,y_4\}$ is a blocking pair (since $\{f_2,f_3\}$ is 2-balanced). Moreover, $x_1 \neq y_4$ (and symmetrically, $x_4 \neq y_1$), for otherwise $U_1=\{f_1\}$ and $\{y_4,x_2\}$ is a blocking pair.
If $U_1$ and $U_4$ contain no close unbalanced pair, then all edges in $U_1$ have the same colour and all edges in $U_4$ have the same colour. Else we may assume (by symmetry) that there is a close unbalanced pair $\{f_i,f_{i+1}\}$ in $U_1$. Since every $f_2$-cycle intersects $C_{i,i+1}$, we have that $y_1$ is an intermediate vertex.
Since the close unbalanced pairs $\{f_3,f_4\}$ and $\{f_i,f_{i+1}\}$ share a vertex, we have $y_i=y_4$ (i.e. $f_i$ is an edge between $y_1$ and $y_4$). Thus the edges in $U_1-\{f_i\}$ have all the same colour. Now if $U_4 \cup \{f_i\}$ is not 2-balanced for $\Omega'-g$, then $y_4$ is also an intermediate vertex, a contradiction to the fact that at most one of $y_1$ and $y_4$ is an intermediate vertex. It follows that $U_1-\{f_i\}$ and $U_4 \cup \{f_i\}$ are both
equivalence classes for $\sim$. Thus we may assume that we started with $U_1$ and $U_4$ where all the edges in $U_1$ have the same colour and all the edges in $U_4$ have the same colour.

Now if $U_1 \cup U_4$ have the same colour, then $\Omega$ is projective planar with a special vertex.
So we may assume that $U_1$ and $U_4$ have different colours.
Such colours are both different from the colour of $f_2$ and $f_3$ (since $\{f_1,f_2\}$ and $\{f_3,f_4\}$ are close unbalanced pairs).
We show that in this case $\Omega$ is a tricoloured graph.
Denote as $N_1$ the neighbour of $y_1$ via edges in $U_1$ and as $N_4$ the neighbours of $y_4$ via edges in $U_4$.
Since $U_1$ and $U_4$ have different colours, there is a 2-vertex-cut $\{d,d'\}$ of $\Omega'$ with $d\in C[y_4,\ldots,y_m,x_1]$ and $d'\in C[x_4,\ldots,x_m,y_1]$, where $\{d,d'\}$ separates $y_1$ from $N_1$ and $y_4$ from $N_4$.
Now $\Omega$ is a tricoloured graph where (following the notation in the definition of tricoloured graphs given in \sref{tricoloured}) we have $I=\{1,3,5\}$ and:
\begin{itemize}
    \item $H_1$ is the $\{c,c'\}$-bridge containing $x_2$;
    \item $H_2=\{c'\}$ if $c'=d'$, otherwise $H_2$ is the $\{c',d'\}$-bridge of $\Omega'$ not containing $x_2$ (or $H_2=\{c'd'\}$ if there is only one $\{c',d'\}$-bridge);
    \item $H_3=\{d'\}$ if $y_2 = d'$, otherwise $H_3$ is the $\{d',y_2\}$-bridge of $\Omega'$ not containing $x_2$ (or $H_3=\{d'y_2\}$ if there is only one $\{d',y_2\}$-bridge);
    \item $H_4=\{g\}$;
    \item $H_5=\{d\}$ if $y_3 = d$, otherwise $H_5$ is the $\{y_3,d\}$-bridge of $\Omega'$ not containing $x_2$ (or $H_5=\{y_3d\}$ if there is only one $\{y_3,d\}$-bridge);
    \item $H_6=\{c\}$ if $c=d$, otherwise $H_6$ is the $\{d,c\}$-bridge of $\Omega'$ not containing $x_2$ (or $H_6=\{dc\}$ if there is only one $\{d,c\}$-bridge).
\end{itemize}

\textbf{Case 3:} $U-\{f_2,f_3\}$ partitions into two sets $U_1$ and $U_4$, where all the edges in $U_1$ are incident to $y_1$ and all the edges in $U_4$ are incident to $x_4$. Since $\{x_2,x_4\}$ is not a blocking pair, $y_1 \neq x_4$.
First suppose that $U_4$ contains a close unbalanced pair $\{f_i,f_{i+1}\}$.
Since $C_{i,i+1}$ and $C_{1,2}$ intersect, this implies that $f_{i+1}$ is an edge between $x_4$ and $x_1$. Thus $U_1=\{f_1\}$ and we are back to Case 1.
So we may assume that all the edges in $U_4$ have the same colour.
Similarly, we either reduce to Case 2 or we may assume that all the edges in $U_1$ have the same colour.

Now if $U_1$ and $U_4$ have the same colour, then $\Omega$ is projective planar with a special vertex.
Otherwise, $U_1$ and $U_4$ have distinct colours (which are also distinct from the colour of $\{f_2,f_3\}$).
Denote as $N_1$ the neighbour of $y_1$ via edges in $U_1$ and as $N_4$ the neighbours of $x_4$ via edges in $U_4$.
Since $U_1$ and $U_4$ have different colours, there is a 2-vertex-cut $\{d,d'\}$ of $\Omega'$ with $d\in C[y_4,\ldots,y_m,x_1]$ and $d'\in C[x_4,\ldots,x_m,y_1]$, where $\{d,d'\}$ separates $y_1$ from $N_1$ and $x_4$ from $N_4$.
Now we can see (similarly to the previous cases) that $\Omega$ is a tricoloured graph.
\end{cproof}
\end{proof}

\begin{proof}[\textbf{Proof of \tref{main-thm}}]
Let $\Omega=(G,\bal)$ be a simple tangled biased graph.
By \lref{small} we may assume that $\Omega$ has at least six vertices.
If $\Omega$ is not 4-connected, then the result holds by \lref{4conn}.
Thus for the remainder of the proof we will assume that $\Omega$ is 4-connected.
By \lref{planar} we only need to consider the case that $\Omega$ is projective planar, that is, $\Omega$ contains a
$2$-connected spanning balanced subgraph $\Omega'$ such that
$(\Omega',(x_1,\ldots,x_m,y_1,\ldots,y_m))$ is planar (where consecutive vertices may be repeated), 
where $E(\Omega)-E(\Omega')=\{x_iy_i\ |\ i \in [m]\}$.
By \lref{hasbp} we may assume that $\Omega$ has a blocking pair $w_1, w_2$. 

Next we show that there is a balanced spanning subgraph $\Omega''$ of $\Omega$ such that all edges in $E(\Omega)-E(\Omega'')$ are incident to $\{w_1,w_2\}$ and such that, up to swapping $w_1$ with $w_2$, $\Omega''$ satisfies one of the following:
\begin{itemize}
	\item $\Omega''$ is 2-connected, or
	\item $\delta_{\Omega''}(w_1)=\{w_1q\}$ for $q\neq w_2$ and $\Omega''/w_1q$ is 2-connected, or
	\item $\delta_{\Omega''}(w_1)=\{w_1q_1\}$,  $\delta_{\Omega''}(w_2)=\{w_2q_2\}$, where $w_1,w_2,q_1,q_2$ are all distinct and $\Omega''/\{w_1q_1, w_2q_2\}$ is 2-connected.
\end{itemize}

Set $E(\Omega)-E(\Omega')=U$. Let $U'$ be the set of edges in $U$ not incident with $\{w_1,w_2\}$.
If $U'$ is empty then we are done, choosing $\Omega''=\Omega'$.

Let $C$ be the facial boundary cycle containing $x_1,\ldots,x_m,y_1,\ldots,y_m$.
For every $f=xy \in U'$, every $(x,y)$-path in $\Omega'$ intersects $\{w_1,w_2\}$. Thus (up to swapping $w_1$ with $w_2$) either for some $i \in [m]$, $w_1 \in C[x_{i},x_{i+1}]$ and $w_2 \in C[y_i,y_{i+1}]$ or $w_1 \in C[y_m,x_1]$ and $w_2 \in C[x_m,y_1]$. By possibly relabelling $x_1,\ldots,x_m$ and $y_1,\ldots,y_m$ we may assume that $w_1 \in C[y_m,x_1]$ and $w_2 \in C[x_m,y_1]$. It follows that $\Omega'$ contains a 2-separation $(A,B)$, where the boundary of $A$ is $\{w_1,w_2\}$ and $C[x_1,\ldots,x_m]$ is contained in $\Omega'[A]$ and $C[y_1,\ldots,y_m]$ is contained in $\Omega'[B]$. 
For $i=1,2$, let $A_i$ be the set of edges in $A$ incident with $w_i$. 
Note that $\Omega-\{w_1,w_2\}=(\Omega'\cup U')-\{w_1,w_2\}$ is 2-connected, since $\Omega$ is 4-connected. Moreover, $\Omega' \cup U' - (A_1 \cup A_2)$ is balanced, since $w_1, w_2$ is a blocking pair.
Now we may choose $\Omega''=\Omega' \cup U' - (A_1 \cup A_2)$. The desired connectivity properties for $\Omega''$ hold, since $\Omega'$ is 2-connected (with different cases occurring when there is only one edge in $B$ incident with $w_i$ or not). By our choice of $\Omega''$ we have that all edges in $E(\Omega)-E(\Omega'')$ are incident with $\{w_1,w_2\}$.

With this structure in place, we show that $\Omega$ is projective planar with either a special pair or a special triple.

Set $U''= E(\Omega)-E(\Omega'')$.
Let $U_1$ be the set of edges in $U''$ incident with $w_1$ but not with $w_2$ and
$U_2$ be the set of edges in $U''$ incident with $w_2$ but not with $w_1$.
Let $W_1$ be the set of neighbours of $w_1$ via edges in $U_1$ and $W_2$ be the set of neighbours of $w_2$ via edges in $U_2$. Then $(\Omega'',(w_1,w_2,W_1,W_2))$ is planar.

Therefore, either $\Omega'' \cup U_1$ and $\Omega'' \cup U_2$ are signed graphs (and $\Omega$ is projective planar with a special pair), or we may assume that there exist edges $e_1=w_1z_1$ and $e_2=w_1z_2$ in $U_1$ such that every $\{e_1,e_2\}$-cycle for $\Omega''$ is unbalanced. We may choose $e_1$ and $e_2$ so that $w_1,w_2,z_1,z_2$ appear in this order in the outer face of $\Omega''$ and let $z_2$ be in $W_1-W_2$ if possible. Note that $z_1\neq z_2$ for otherwise $\{z_1,z_2\}$ is  a cut of $\Omega$. 

Suppose that there exists a vertex $z \in W_2-\{z_2\}$. Then one of the following occurs: $\Omega''$ contains a $(z_1-z_2,z-w_2)$-linkage avoiding $w_1$,
or there is a 2-vertex-cut $\{w_1,c\}$ in $\Omega''$ separating $z$ from $w_2$ and $z_1$ from $z_2$.
The former case does not occur, else $\Omega$ contains two disjoint unbalanced cycles.
Suppose that the latter occurs.
Since $\{w_1,w_2,c\}$ is not a 3-vertex-cut of $\Omega$, we have that $c=z_1$ and
$w_2$ is only incident with $w_1$ and $z_1$ in $\Omega''$. Hence, $U_1-\{e_1\}$ is 2-balanced. If $U_2$ is not 2-balanced, then $w_2z_2\in U_2$, and there is a vertex $z'_2\in W_2$ such that every $\{w_2z_2, w_2z'_2\}$-cycle for $\Omega''$ is unbalanced and $w_1,w_2,z_1,z_2,z'_2$ appear in this order in the outer face of $\Omega''$. By symmetry $w_1$ is only incident with $w_2$ and $z'_2$ in $\Omega''$, and $U_2-\{w_2z'_2\}$ is 2-balanced. Moreover, since $\Omega$ has no two disjoint  unbalanced cycles, $W_1=\{z_1,z_2\}$ and $W_2=\{z_2,z'_2\}$; so either $\{z_1,z_2,z'_2\}$ is a cut of $\Omega$ or $V(\Omega)=\{w_1,w_2,z_1,z_2,z'_2\}$, a contradiction.   So $U_2$ is 2-balanced, implying that $\Omega''\cup U_2 -\{w_1w_2,w_2z_1\}$ is balanced. In this case $\Omega$ is projective planar with a special triple, by setting (in the definition of projective planar graph with a special triple), $H=\Omega''\cup U_2 -\{w_1w_2,w_2z_1\}$ and $x=w_1, y_1=w_2$ and $y_2=z_1$.

The remaining case is when $W_2=\{z_2\}$ and, if there exists any other edge $e \in U_1$ such that the $\{e,e_1\}$-cycles for $\Omega''$ are unbalanced, then $e$ is also incident with $z_2$ by the choice of $z_2$. In this case $\Omega$ is projective planar with a special triple, by setting  (in the definition of projective planar graph with a special triple) $x=w_1$, $y_1=z_2$ and $y_2=w_2$.
\end{proof}

\section*{Aknowledgements}
We would like to thank Geoff Whittle for his support, in particular for the first author's visit to the University of Western Australia, where this project was initiated.

\appendix

\section{Pictures of tangled structures}\label{sec:figures}
Following are pictures of the structures described in \sref{structures}.
The grey subgraphs represent balanced graphs; these balanced graphs are all planar, with the exception of Figure~\ref{fig:fat-triangle}.

\begin{figure}[h!]
\begin{center}
\includegraphics[page=2,height=5.2cm]{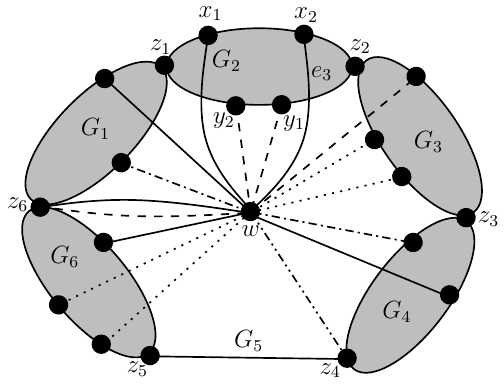}
\caption{A generalised wheel.}
\label{fig:gen-wheel}
\end{center}
\end{figure}

\begin{figure}[h!]
\begin{minipage}[b]{0.45\linewidth}
\centering
\includegraphics[page=3,height=4.5cm]{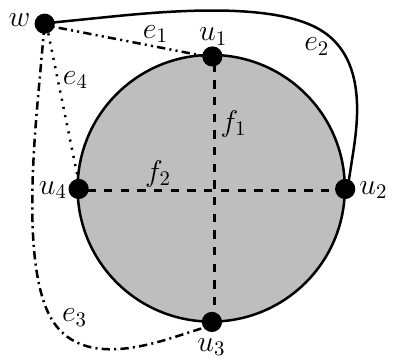}
\caption{A criss-cross.}
\label{fig:criss-cross}
\end{minipage}
\hspace{0.5cm}
\begin{minipage}[b]{0.45\linewidth}
\centering
\includegraphics[page=10,height=4cm]{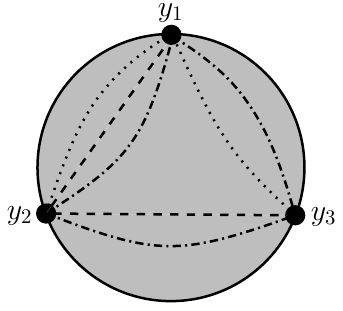}
\caption{A fat triangle.}
\label{fig:fat-triangle}
\end{minipage}
\end{figure}

\begin{figure}[h!]
\begin{center}
\includegraphics[page=4,height=4.5cm]{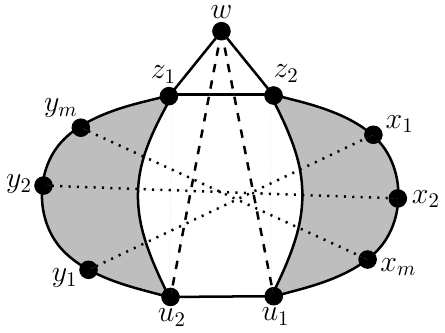}
\includegraphics[page=5,height=4cm]{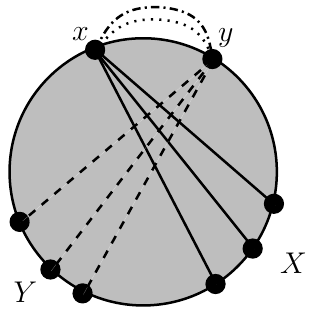}
\includegraphics[page=6,height=4cm]{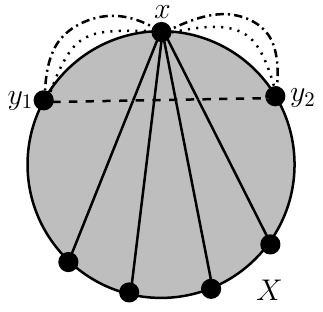}
\caption{A projective planar biased graph with, from left to right, a special vertex, a special pair and a special triple.}
\label{fig:ppsv}
\end{center}
\end{figure}

\begin{figure}[h!]
\begin{center}
\includegraphics[page=11,height=6cm]{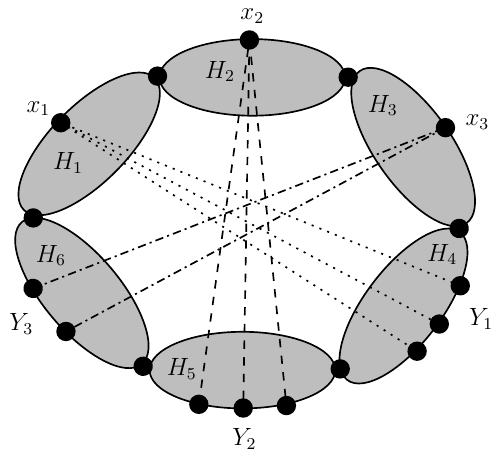}
\includegraphics[page=12,height=6cm]{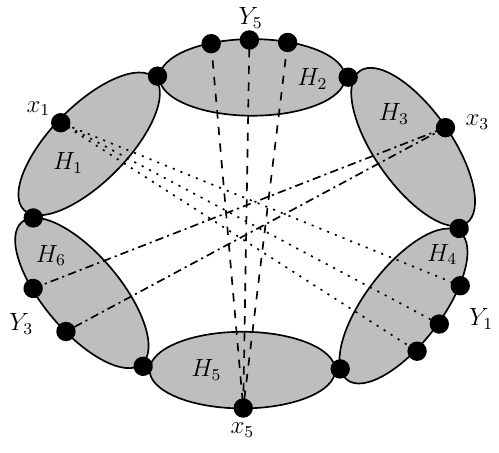}
\caption{Tricoloured biased graphs.}
\label{fig:tric}
\end{center}
\end{figure}

\clearpage
\bibliographystyle{amsplain}
\bibliography{/Users/piv8irene/Dropbox/Irene-bibliography}

\end{document}